\documentclass[12pt]{article}
\usepackage[pagebackref]{hyperref}
\usepackage{color}
\definecolor{myurlcolor}{rgb}{0,0,0.4}
\definecolor{mycitecolor}{rgb}{0,0.5,0}
\definecolor{myrefcolor}{rgb}{0.5,0,0}
\hypersetup{colorlinks,
linkcolor=myrefcolor,
citecolor=mycitecolor,
urlcolor=myurlcolor}
\usepackage{amsmath,amsthm,amssymb,amsfonts,amsrefs}
\usepackage{mathtools}
\usepackage{latexsym}
\usepackage{tikz}
\usepackage{amsrefs}
\usepackage[all,arc,dvips,arrow,curve,cmactex,2cell,rotate]{xy}
\usepackage{amsfonts,graphics}
\usepackage{cancel}
\usepackage[stable]{footmisc} 
\usepackage{enumitem}
\usepackage{phoenician}
\usepackage{microtype}
\usepackage[left=2cm,right=2cm,top=2cm,bottom=2cm]{geometry}
\usepackage{mleftright}
\usepackage{subfigure}
\usepackage{cancel}

\makeatletter \@addtoreset{equation}{section} \makeatother

\swapnumbers 
\newtheorem{Theorem}{Theorem}[subsection]
\newtheorem{Definition}[Theorem]{Definition}
\newtheorem{Lemma}[Theorem]{Lemma}

\newtheorem{Proposition}[Theorem]{Proposition}

\theoremstyle{definition}
\newtheorem{Remark}[Theorem]{Remark}
\newtheorem{Example}[Theorem]{Example}

\def\be{\begin{equation}}
\def\ee{\end{equation}}
\def\ba{\begin{eqnarray}}
\def\ea{\end{eqnarray}}
\newcommand{\R}{\mathbb{R}}
\newcommand{\N}{\mathbb{N}}
\newcommand{\Nl}{\mathbb{N}}
\newcommand{\Rl}{\mathbb{R}}
\newcommand{\defin}{:=}  
\newcommand{\C}{\mathcal{C}}
\newcommand{\id}{\mathrm{id}}

\newcommand{\vxymatrix}[1]{\vcenter{\vbox{\xymatrix{#1}}}}
\newcommand{\Sets}{\mathsf{Sets}}
\newcommand{\Cats}{\mathsf{Cats}}
\newcommand{\FinSets}{\mathsf{FinSets}}
\newcommand{\Tops}{\mathsf{Tops}}
\newcommand{\Mets}{\mathsf{Mets}}
\newcommand{\Top}{\mathsf{Top}}
\newcommand{\sC}{\mathsf{C}}
\newcommand{\Sp}{\mathsf{Sp}}
\newcommand\op{{\rm op}}

\title{{\sf Compositories and Gleaves}}

\author{{\sf Cecilia Flori\thanks{{\sf cflori@waikato.ac.nz}}, Tobias Fritz}\thanks{{\sf fritz@mis.mpg.de}}\\
\\
{\sf Perimeter Institute for Theoretical Physics}}
\date{\sf \today}

\begin{document}

\maketitle

\begin{abstract}
Sheaves are objects of a local nature: a global section is determined by how it looks locally. Hence, a sheaf cannot describe mathematical structures which contain global or nonlocal geometric information. To fill this gap, we introduce the theory of ``gleaves'', which are presheaves equipped with an additional ``gluing operation'' of compatible pairs of local sections. This generalizes the \emph{conditional product} structures of Dawid and Studen\'y, which correspond to gleaves on distributive lattices. Our examples include the gleaf of metric spaces and the gleaf of joint probability distributions. A result of Johnstone shows that a category of gleaves can have a subobject classifier despite not being cartesian closed. 

Gleaves over the simplex category $\Delta$, which we call compositories,  can be interpreted as a new kind of higher category in which the composition of an $m$-morphism and an $n$-morphism along a common $k$-morphism face results in an $(m+n-k)$-morphism. The distinctive feature of this composition operation is that the original morphisms can be recovered from the composite morphism as initial and final faces. Examples of compositories include nerves of categories and compositories of higher spans.
\end{abstract}

\tableofcontents

\section{Introduction}

\paragraph{Sheaves and stacks.}

Sheaves on topological spaces have the defining property that compatible families of local sections can always be glued together in a unique way. If the given local sections are defined over an open covering of the space, then this yields a global section. It follows that sheaves are entirely \emph{local} entities: a sheaf on a topological space is completely determined by its stalks and how those assemble into the associated \'etale bundle~\cite{MM}. This makes sheaves the right framework for mathematical structures which are of a local nature, in the sense that the only global information contained in a sheaf is of topological character. Similar statements apply to sheaves on sites.

A more sophisticated variant of the notion of sheaf, which enjoys similar locality properties, is the notion of \emph{stack} or \emph{2-sheaf}. Many mathematical structures like vector bundles or principal bundles can be described in terms of a stack. Again, the crucial property here is \emph{locality}: a vector bundle or principal bundle is an entirely local object, meaning that it is also determined by what it looks like on an open cover.

\paragraph{A sheaf of metrics?}

For many other mathematical structures, however, this kind of locality property fails. For example, consider the set of all metrics $d:X\times X\to \R_{\geq 0}$ on a set $X$. Any such metric can be restricted to any subset $U\subseteq X$, and hence we obtain a presheaf $2^X\to\Sets$ which assigns to each subset $U\subseteq X$ the set of all metrics on $U$. However, this presheaf is not a sheaf: in general, there are many ways to extend a pair of metrics
\be
d_U: U\times U\to \R_{\geq 0},\qquad d_V:V\times V\to \R_{\geq 0} ,
\ee
which are compatible in the sense that $d_{U}|_{U\cap V}=d_{V}|_{U\cap V}$, to a metric on $U\cup V$. For example, we may define
\be
\label{gluemetrics}
d(u,v) \defin \inf_{x\in U\cap V} \mleft( d_U(u,x) + d_V(x,v) \mright) 
\ee
for $u\in U\setminus V$ and $v\in V\setminus U$, while retaining the given distances inside $U$ and $V$. As we will see, this is indeed a (pseudo-)metric on $U\cup V$. While this is a canonical extension of the given metrics on $U$ and $V$ to one on $U\cup V$, it is by no means unique and many other extensions are possible in general.

On the other hand, for given compatible metrics $d_U,d_V,d_W$ on \emph{three} pairwise intersecting sets $U,V,W\subseteq X$, there may not exist \emph{any} metric on the union $U\cup V\cup W$ which extends all three of them. The reason is that the given distances may fail to satisfy the triangle inequality for a triangle spanning the three pairwise intersections, i.e.~there may exist points
$$
x\in U\cap V,\quad y\in U\cap W,\quad z\in V\cap W \qquad\text{ s.t. }\qquad d_V(x,z) > d_U(x,y) + d_W(y,z) .
$$
(See Figure~\ref{triangle}.)

These considerations show that the presheaf of metrics has a \emph{global} structure which prevents it from being a sheaf. This should be seen in contrast to the case in which $X$ is a smooth manifold and one assigns to each open $U\subseteq X$ the set of \emph{Riemannian} metrics on $U$. In this case, we do indeed get a sheaf, since a Riemannian metric---by virtue of being a tensor field---is determined by local data. For general metrics, this is not the case: simply knowing how a metric looks locally is not sufficient to determine it globally. 

One of the questions that we would like to tackle in this paper is: what is the presheaf of metrics, if not a sheaf? Equation~\eqref{gluemetrics} suggests that it does have more structure than being ``just'' a presheaf. Such a canonical gluing of pairs of local sections occurs not only in the context of metric spaces, but in many other situations as well, and hence this is one of the things that we would like to formalise. Since it is supposed to capture the idea of a sheaf-like structure describing geometry in a \emph{global} way such that certain \emph{gluings} are still possible, we will call such a gadget a \emph{gleaf}.

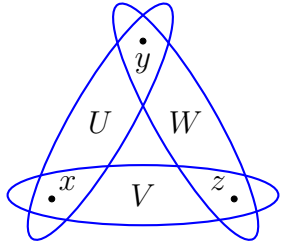
\begin{figure}
\begin{center}
\begin{tikzpicture}
\node[draw,shape=circle,fill,scale=.2] (a) at (90:1.4) {} ;
\node[draw,shape=circle,fill,scale=.2] (c) at (210:1.4) {} ;
\node[draw,shape=circle,fill,scale=.2] (e) at (330:1.4) {} ;
\node[below of=a,node distance=3mm] {$y$};
\node[above right of=c,node distance=3mm] {$x$};
\node[above left of=e,node distance=3mm] {$z$};
\draw[thick,blue,rotate=270] (0:.65) ellipse (.4cm and 1.8cm) ;
\draw[thick,blue,rotate=150] (0:.655) ellipse (.4cm and 1.8cm) ;
\draw[thick,blue,rotate=30] (0:.655) ellipse (.4cm and 1.8cm) ;
\node at (30:.655) {$W$};
\node at (150:.655) {$U$};
\node at (270:.655) {$V$};
\end{tikzpicture}
\end{center}
\caption{Three subsets $U,V,W\subseteq X$ with pairwise intersection.}
\label{triangle}
\end{figure}

We suspect that many other geometrical structures actually form a gleaf. We work this out for the case of topological spaces in Section~\ref{topological}.

\paragraph{Joint probability distributions.}

A very similar situation arises in probability theory~\cite{AB}. In this case, the base set $X$ stands for a collection of random variables. To keep things simple, we assume each variable to take values in the same finite set of outcomes $O$. To a subset of variables $U\subseteq X$, we assign the set of all joint probability distributions $P_U$ of these variables. This assignment turns into a presheaf $2^X\to\Sets$ if one takes the restriction maps to be given by the formation of marginal distributions.

As in the previous example of metrics, compatible triples of local sections on pairwise intersecting subsets $U,V,W\subseteq X$ are often not extendible to a local section on $U\cup V\cup W$. The smallest example occurs for three variables, $X=\{A,B,C\}$, with given subsets the two-variable ones,
\[
U=\{A,B\}, \quad V=\{B,C\},\quad W=\{A,C\},
\]
and binary outcomes $O=\{0,1\}$. Now let $P_U$ stand for perfect correlation between uniformly random $A$ and $B$, and likewise $P_V$ between $B$ and $C$. Since perfect correlation is a transitive relation on random variables, assuming the existence of a joint distribution $P_X$ implies that also $P_W$ corresponds to perfect correlation between $A$ and $C$. Hence, if $P_W$ stands e.g.~for perfect \emph{anti}correlation between $A$ and $C$, then no joint distribution can exist, although the given two-variable distributions are compatible in the sense that they marginalise to the same single-variable distributions. For more detail on such \emph{marginal problems} and further results, see~\cite{AB,FC} and references therein, in particular~\cite{V}.

Nevertheless, again \emph{pairs} of local sections can always be glued together in a canonical way, using a formula vaguely reminiscent of~\eqref{gluemetrics}. If $P_U$ and $P_V$ are given distributions on $U=\{A,B\}$ and $V=\{B,C\}$ which have compatible marginal on $B$,
\be
\label{marginalcomp}
\sum_{a\in O} P_U(a,b) = \sum_{c\in O} P_V(b,c) \qquad\forall b\in O,
\ee
then a canonical joint distribution for $\{A,B,C\}$ is given by
\be
\label{probglue}
P(a,b,c) \:\defin\: \frac{P_U(a,b)P_V(b,c)}{P_{U\cap V}(b)},
\ee
where the denominator term stands for either side of~\eqref{marginalcomp}. Here, it is understood that the left-hand side is declared to be $0$ whenever the denominator vanishes; in this case, also both terms in the numerator vanish. That this joint distribution recovers both $P_U$ and $P_V$ as marginals is easy to see: for example, summing over $a$ turns $P_U(a,b)$ into $P_{U\cap V}(b)$, which cancels with the denominator, so that $P_V(b,c)$ remains. Formula~\eqref{probglue} is natural from the probability point of view: it is precisely that joint distribution which makes $A$ and $C$ conditionally independent given $B$.

There is an immediate generalization to pairs of sets containing any number of variables. For example, if $U\cap V=\emptyset$, then the distribution in the denominator of~\eqref{probglue} is the ``joint'' distribution of no variables, i.e.~equal to the constant $1$, and the overall joint distribution simply becomes the product distribution of $P_U$ and $P_V$.

As we will see, the formal categorical properties of~\eqref{probglue} are entirely analogous to those of~\eqref{gluemetrics}: the presheaf of joint probability distributions of random variables is also a gleaf. This reproduces the \emph{conditional product} structures of Dawid and Studen\'y~\cite{DS} and provides a categorical formulation for them.

A related example of a gleaf is the one formed by relations of finite arity, or equivalently of tables in a relational database (Section~\ref{reldatabase}, also~\cite{DS}).

\paragraph{Gluing as composition.}

There is another point of view on the probability distributions example which hints at a higher categorical structure. We now omit the subscripts on the distributions and simply write $P(a,b)$ for a joint distribution of variables $A,B\in X$.

In the spirit of categorical probability theory, we would like to consider a joint distribution $P(a,b)$ as a \emph{morphism} between the associated marginal distributions $P(a)$ and $P(b)$. In other words, the objects of our category are single-variable distributions $P(a),\: P(b),\:\ldots$, while the morphisms are two-variable distributions such that the marginal of the first variable reproduces the source object and the marginal of the second variable reproduces the target object. Composition of morphisms is defined essentially by~\eqref{probglue}; the only difference is that one needs to take the two-variable marginal of $A$ and $C$ by summing over the possible values of $B$,
\be
\label{Pcomp}
P(a,c)=\sum_b\frac{P(a,b)P(b,c)}{P(b)} .
\ee
Diagrammatically, we thus have
\be
\begin{split}
\label{Pdiag}
\xymatrix{
P(a)\ar[rr]^{P(a,b)}\ar@/_3pc/@{-->}[rrrr]_{P(a,c)}&&P(b)\ar[rr]^{P(b,c)}&&P(c)\\
}\end{split}
\ee
This relates nicely to the usual categories of conditional probability distributions studied in categorical probability theory~\cite{Abram,Dob,Pan}. In the finite case, the morphisms of these categories are typically conditional probability distributions $P(b|a)$, that is stochastic matrices, whose composition is given by matrix multiplication,
$$
P(c|a) = \sum_b P(c|b) \, P(b|a) .
$$
And indeed, this is precisely what one obtains upon dividing both sides of~\eqref{Pcomp} by $P(a)$ and rewriting everything in terms of conditional probabilities.

So far, we have taken single-variable distributions to be objects in our category, while two-variable distributions are morphisms. Now the obvious question is, what about three-variable or $n$-variable distributions? In fact, it seems like the composition of a composable pair of morphisms produce more than just another morphism: if we follow the spirit of~\eqref{probglue}, it seems natural to omit the summation in~\eqref{Pcomp} and consider the three-variable distribution
\be
\label{Pabc}
P(a,b,c)=\frac{P(a,b)P(b,c)}{P(b)} .
\ee
as the ``composition'' of $P(a,b)$ and $P(b,c)$. Recovering the usual composition~\eqref{Pcomp} is simple by just taking a marginal.

\paragraph{Compositories.}

It may now be clear that we will construct a \emph{higher category} of probability distributions in which the $n$-morphisms are $n$-variable joint distributions. Alas, since we want the composition of two $1$-morphisms to be a $2$-morphism
\[\xymatrix{
& P(b)\ar[rd]^{P(b,c)}\ar@{}[]="1" &\\
 P(a)\ar[rr]_{P(a,c)}="2" \ar[ru]^{P(a,b)} && P(c) 
 \ar@{=>}"1";"2"|{P(a,b,c)}
}\]
this ``higher category'' cannot be a higher category in any of the usual senses of the word~\cite{Leinster2}. In fact, we will see that the composition of an $m$-morphism with an $n$-morphism along a common $k$-morphism gives an $(m+n-k)$-morphism. In order to emphasise this conceptual departure from (higher) categories, we will call such gadgets \emph{compositories}. A compository will be a simplicial set equipped with a certain compositional structure satisfying a couple of axioms; as shown in Theorem~\ref{gleavescompo}, they can also be regarded as gleaves over the simplex category $\Delta$. We give a precise definition in Section~\ref{seccomp}.

Natural examples of compositories also arise in category theory: the nerve of a category is a compository in a unique way (Section~\ref{nerves}), and higher spans in a category with pullbacks assemble into a compository (Section~\ref{higherspans}).

\paragraph{Disclaimer.} In the course of our investigations, we have completely changed our basic definitions several times over. While we now seem to have reasonable axioms giving rise to a nice abstract theory, it is not completely clear whether the structures that we study are indeed fundamental, and whether our definitions are appropriate in terms of the precise formulations of the technical details. In fact, we have some indications suggesting that at least the latter is not yet the case; for example, the somewhat strange-looking Definition~\ref{morphbicov}. In this sense, the present paper should be regarded as an exposition of some preliminary definitions which may be subject to change.

\paragraph{Summary and structure of this paper.}

We begin in Section~\ref{simplicialsets} with a brief recap of simplicial sets and fix the corresponding notation. Section~\ref{compositoryintro} introduces compositories as simplicial sets equipped with a composition operation which turns an $m$-simplex and an $n$-simplex into an $(m+n-k)$-simplex, provided that the two given simplices have a common $k$-simplex face. This composition is required to satisfy certain axioms from which we derive a number of consequences, including associativity of composition. We then give various examples of compositories. This starts with nerves of categories in Section~\ref{nerves}; despite being a neat example, we also nerves of generic 2-categories cannot be equipped with a compository structure. Other examples are compositories of higher spans in Section~\ref{higherspans}, the compository of metric spaces in Section~\ref{metspacesI} and the compository of joint probability distributions in Section~\ref{secjpd}. 

We then depart from the study of compositories and introduce the notion of a gleaf on a lattice in Section~\ref{gleaveslattice}, based on the observation that certain presheaves, despite not being sheaves, still have canonical gluings of compatible pairs of local sections. In order to generalise this from gleaves on a lattice to gleaves on a category, Section~\ref{bicoverings} is dedicated to the development of the notion of \emph{system of bicoverings} which relates to gleaves on a category as the notion of Grothendieck topology relates to sheaves on a category. Section~\ref{categorybicoverings} discusses gleaves on a category with bicoverings. Not only do sheaves on a site turn out to be gleaves, but also compositories are the same as gleaves on the simplex category $\Delta$.  Section~\ref{categorygleaves} defines morphisms of gleaves and explains a result of Johnstone showing that the category of gleaves over a certain base has a subobject classifier despite not being cartesian closed.
Concrete examples of gleaves are the gleaf of metric spaces (Section~\ref{takeII}), the gleaf of joint probability distributions (Section~\ref{take2}), a gleaf used in relational database theory (Section~\ref{reldatabase}), and the gleaf of topological spaces (Section~\ref{topological}). In the first three cases, the gluing operation is related to well-known concepts: shortest paths in metric spaces, conditional independence of random variables, and the join operation from relational database theory.

We conclude with a list of further directions in Section~\ref{further}.

\paragraph{Acknowledgements.}

We would like to thank Samson Abramsky, Paolo Bertozzini, Jonathan Elliott, Peter Johnstone, Peter LeFanu Lumsdaine, Prakash Panangaden, David Roberts, Michael Shulman, Rui Soares Barbosa and Ross Street for helpful feedback and suggestions.

Research at Perimeter Institute is supported by the Government of Canada through Industry Canada and by the Province of Ontario through the Ministry of Economic Development and Innovation. The second author has been supported by the John Templeton Foundation.

\section{Compositories}
\label{seccomp}

\subsection{Background on simplicial sets}\label{simplicialsets}

Before introducing compositories, we recall some background on simplicial sets and fix some notation; see also~\cite{Fried,Riehl}. Simplicial sets are among the basic \emph{geometric shapes} for higher categories~\cite{Joyal,Leinster}.

The \emph{simplex category}\footnote{This category is sometimes also called the \emph{topologist's} simplex category in order to distinguish it from the \emph{augmented} simplex category of Remark~\ref{augmented}, which is also known as the \emph{algebraist's} simplex category.} $\Delta$ has as objects the non-empty finite ordinals
\[
	[n] = \left\{ 0,\ldots, n\right\} \qquad (n\in\Nl),
\]
equipped with the standard ordering, and as hom-sets $\Delta([n],[m])$ all the order-preserving functions $[n]\to [m]$. Composition is ordinary composition of functions. Then, a \emph{simplicial set} is a presheaf $\Delta^{\op}\to\Sets$.

For $k\in[n]$, an important class of morphisms in $\Delta$ is given by the \emph{face maps}
\begin{align}
\begin{split}
\partial_k:[n-1]&\to [n]\\
v&\mapsto \begin{cases}v & \text{ if } v<k\\ v+1 & \text{ if } v\geq k,\end{cases}
\end{split}
\end{align}
and the \emph{degeneracy maps}
\begin{align}
\begin{split}
\eta_k:[n+1] & \to [n]\\
v&\mapsto\begin{cases} v & \text{ if }v\leq k\\ v-1 & \text{ if } v>k. \end{cases}
\end{split}
\end{align}
In both cases, we suppress the dependence on $n$ in our notation, since it is determined by the object that $\partial_k$, respectively $\eta_k$, is applied to.

It is well-known that the face and degeneracy maps generate $\Delta$ and present it via the relations

\begin{align}
\begin{array}{rcll}
\partial_k\partial_j &=& \quad\partial_j  \partial_{k-1} &
\hspace{-3.4cm}\text{if }j<k, \smallskip\\[2pt]
\eta_k\eta_j &=& \quad \eta_j \eta_{k+1} & \hspace{-3.4cm}\text{if
}j\leq k, \smallskip\\[2pt]
\eta_k\partial_j &=& \begin{cases}\partial_j \eta_{k-1} &\qquad\text{if
} j<k, \smallskip\\
 \id_{[n]} & \qquad\text{if } j\in\{k,k+1\}, \smallskip\\
  \partial_{j-1}\eta_k & \qquad\text{if } j>k+1.\end{cases}
\end{array}
\end{align}

For convenience of notation, we also introduce the family of morphisms $s_k \defin\partial_n\cdots\partial_{k+1}$, or explicitly 
\begin{align}
\begin{split}
s_k:[k] &\to [n]\\
v&\mapsto v .
\end{split}
\end{align}
For presheaves on $\Delta$, restricting along $s_k$ corresponds to taking the initial $k$-face (``source'') of an $n$-simplex in a simplicial set $\Delta^{\mathrm{op}}\to\Sets$. Similarly, we work with the family of morphisms $t_k=\partial_0\cdots\partial_0$, with $n-k$ factors, or explicitly
\begin{align}
\begin{split}
t_k:[k]&\to [n] \\
v&\mapsto v+n-k.
\end{split}
\end{align}
On the presheaf level, $t_k$ takes the terminal $k$-face (``target'') of every $n$-simplex. Again, the $n$-dependence of $s_k$ and $t_k$ is left implicit.

If $\C:\Delta^{\op}\to\Sets$ is a simplicial set, $f:[m]\to[n]$ and $A\in\C(n)$, then we write $Af$ as shorthand for $\C(f)(A)$ in analogy with the standard notation for the right action of a group on a set.

\subsection{Definition and first properties of compositories}\label{compositoryintro}

In the following, we use the terms ``morphism'' and ``simplex'' interchangeably.  

\begin{Definition}
Let $\C:\Delta^{\op}\to\Sets$ be a simplicial set. A pair of simplices $(A,B)\in \C(m)\times \C(n)$ is \emph{$k$-composable} for $k\in\N$ if the terminal $k$-face of $A$ coincides with the initial $k$-face of $B$,
\be
\label{kcomposable}
A t_k= Bs_k.
\ee
\end{Definition}

For $m=n=1$ and $k=0$, this recovers the usual notion of composability in a category.

The following axioms for compositories constitute a minimal set of formal requirements. We will use them afterwards to derive consequences with a more intuitive meaning.

\begin{Definition}\label{def:comp}
A \emph{compository} is a simplicial set $\C:\Delta^{\mathrm{op}}\rightarrow \Sets$ equipped with a \emph{composition}
$$
A\circ_k B \:\in\: \C(m+n-k)
$$
for every $k$-composable pair $(A,B)\in \C(m)\times \C(n)$ and every $k\in\N$,  such that the following axioms hold:
\begin{enumerate}
\item \emph{Identity axiom:} Composing a morphism $A$ with a source or target face of itself recovers the morphism:
\begin{align}
\begin{split}
\label{comp}
As_k \circ_k A &=A,\\
A \circ_k At_k &=A.
\end{split}
\end{align}
\end{enumerate}
Moreover, for every $k$-composable pair $(A,B)\in\C(m)\times\C(n)$:
\begin{enumerate}[resume]
\item\emph{Back-and-forth axiom:}
\begin{enumerate}
\item For $i\leq n-k$, composing the source $(m+i)$-face of the composition $A\circ_k B$ with $B$ recovers the composition,
\be
\label{bf1}
 (A\circ_k B)s_{m+i} \circ_{k+i} B = A\circ_k B .
\ee
\item For $j\leq m-k$, composing $A$ with the target $(n+j)$-face of the composition $A\circ_k B$ recovers the composition,
\be
\label{bf2}
 A \circ_{k+j} (A\circ_k B)t_{n+j} = A\circ_k B.
\ee
\end{enumerate}
\item\label{compdeg} \emph{Compatibility with degeneracy maps:}
\begin{enumerate}
\item For $i\leq m-k$,
\be
\label{compdeg1}
(A\circ_k B)\eta_i = A\eta_i \circ_k B,
\ee
\item For $i\geq m$,
\be
\label{compdeg2}
(A\circ_k B)\eta_i = A\circ_k B\eta_{i-m+k},
\ee
\item For $m-k\leq i\leq m$,
\be
\label{compdeg3}
(A\circ_k B)\eta_i =  A\eta_i \circ_{k+1} B\eta_{i-m+k} .
\ee
\end{enumerate}
\item\label{compface} \emph{Compatibility with face maps:}
\begin{enumerate}
\item For $i<m-k$,
\be
\label{faceS}
(A\circ_k B)\partial_i=A\partial_i \circ_k B,
\ee
\item For $i>m$,
\be
\label{faceT}
(A\circ_k B)\partial_i=A \circ_k B\partial_{i-m+k}.
\ee
\end{enumerate}
\end{enumerate}
\end{Definition}

\begin{figure}
	\centering
	\subfigure[$(A\circ_1 B)\eta_1 = A\eta_1\circ_1 B$ from \eqref{compdeg1} with $m=1$ and $n=1$.]{
		\xymatrix{	1 \ar@{=}[drr] &	\\
				& & 2 \ar[rr] && 3		\\
				0 \ar[uu] \ar[urr] & \ar@{}[d]|{\textstyle A\eta_1} &	& \ar@{}[d]|{\textstyle B}	\\	
				& & & & }
		}
	\hspace{2pc}
	\subfigure[$(A\circ_1 B)\partial_1 = A\partial_1 \circ_1 B$ from \eqref{faceS} with $m=3$ and $n=2$.]{
		\hspace{1pc}
		\xymatrix{	0 \ar[ddrr]|\hole \ar[rr] \ar[dd]	&&	2 \ar[dd] \ar[drr]	\\
				& & & & 4				&				\\
				\xcancel{1} \ar[rr] \ar[uurr]		& \ar@{}[d]|{\textstyle A\partial_1} &	3 \ar[urr] & \ar@{}[d]|{\textstyle B}	\\
				& & & }
		}	
	\caption{Illustration of Axioms~\ref{compdeg} and~\ref{compface}.}
	\label{natillu}
\end{figure}

Hereby, the $k$-composability assumption on $A$ and $B$ guarantees that all expression in~(\ref{bf1})--(\ref{faceT}) make sense.

\begin{Remark}
In axiom~\ref{compdeg}, both the first and the third case apply to $i=m-k$, so that we can conclude
\be
B\eta_{0} \circ_{k+1} A\eta_{m-k}=A\eta_{m-k}\circ_k B.
\ee
Similarly, both the second and third case apply to $i=m$, and hence
\be
A\eta_m\circ_{k+1} B\eta_{k} = A\circ_k B\eta_k .
\ee
On the other hand, Axiom~\ref{compface} about the compatibility of composition with face maps says nothing at all about the range $m-k\leq i\leq n$, i.e.~when the face map acts on the common face. This is because the putative condition
\be
\label{facenot}
(A\circ_k B)\partial_i \stackrel{?}{=} A\partial_i\circ_{k-1}B\partial_{i-m+k}
\ee
does \emph{not} hold in many of our examples; in fact, it will only hold in the example of nerves of categories (Section~\ref{nerves}).
\end{Remark}

\begin{Remark}
Degenerate simplices in a compository can be thought of as (higher) \emph{identity morphisms}. Equations~\eqref{compdeg1} and~\eqref{compdeg2} state that composing any morphism with a higher identity results in a higher identity.

In more detail, let us consider the degenerate simplex $F\eta_k$ associated to the terminal $k$-face $F\defin At_k$. Since $F\eta_ks_k=F$, we can form the composition $A\circ_k F\eta_k$, and properties~\eqref{comp} and~\eqref{compdeg2} guarantee that
\be
\label{higherid}
A\circ_k F\eta_k = (A\circ_k F)\eta_m = A\eta_m.
\ee
Together with $\eta_m\partial_m = \id = \eta_m\partial_{m+1}$, this implies that
\[
\mleft(A\circ_k F\eta_k\mright)\partial_m = A\eta_m\partial_m = A ,
\]
and likewise
\[
\mleft(A\circ_k F\eta_k\mright)\partial_{m+1} = A\eta_m\partial_{m+1} = A .
\]
In other words, since the simplex~\eqref{higherid} is degenerate over $A$, it has $A$ as two of its faces. Intuitively,~\eqref{higherid} means that $F\eta_k$ can be thought of as an identity over $F$.

A similar statement can be made for an initial face $E\defin As_k$,
\be
E\eta_0 \circ_k A = A\eta_0 .
\ee
In summary, composing with a (higher) identity of a face results in a (higher) identity.
\end{Remark}

We continue with the derivation of some consequences of the axioms, such as associativity of composition. In the following three lemmas, $A\in\C(m)$ and $B\in\C(n)$ are still assumed to be $k$-composable.

\begin{Lemma}
If $i\geq m$ and $j\geq n$, then
\be
\label{stcomp}
(A\circ_k B)s_i = A\circ_k Bs_{i-m+k},\qquad (A\circ_k B)t_j =
At_{j-n+k}\circ_k B.
\ee
\end{Lemma}

\begin{proof}
Repeated application of~\eqref{faceS} and~\eqref{faceT}.
\end{proof}

\begin{Lemma}[Source and target equations]
The original morphisms $A$ and $B$ can be recovered from the composition $A\circ_k B$ as source and target faces:
\begin{align}
\begin{split}
\label{comp2}
(A\circ _k B)s_m &=A,\\
(A\circ _k B)t_n &=B.
\end{split}
\end{align}
\end{Lemma}

\begin{proof}
	We prove the first equation only; the proof of the first equation is analogous. Using~\eqref{stcomp} gives
	\[
		(A\circ_k B)s_m = A\circ_k Bs_k = A\circ_k At_k = A,
	\]
	where the second step is by composability and the third by~\eqref{comp}.
\end{proof}

In this sense, composition in a compository does not lose any information: the original morphisms can be recovered from their composite. We regard this as one of the main features of compositories not shared by other structures such as categories.

\begin{Lemma}[two-step rule]
A composition $A\circ_k B$ can be computed in two steps: for any $k\leq i\leq m$ and $k\leq j\leq n$,
\begin{align}
\begin{split}
\label{2stepcomp}
A\circ_k B &= A\circ_{i} (At_{i} \circ_{k} B) \\
&= (A\circ_k Bs_{j})\circ_{j} B .
\end{split}
\end{align}
\end{Lemma}

\begin{proof}
To see the first equation, we apply~\eqref{bf2} with $j=i-k$
and~\eqref{stcomp},
\[
A\circ_k B = A\circ_i (A\circ_k B)t_{n+i-k} = A\circ_i (At_i \circ_k B).
\]
The proof of the other equation is analogous.
\end{proof}

While associativity is an axiom for many other mathematical structures containing a binary operation, for compositories it is actually a \emph{derived} property:

\begin{Proposition}[Associativity] For a triple $(A,B,C)\in \C(l)\times \C(m)\times C(n)$ which is $(j,k)$-composable in the sense that
$$
At_j = Bs_j,\qquad Bt_k=Cs_k,
$$
it holds that
\be\label{associat}
A\circ_j (B\circ_k C)=(A\circ_j B)\circ_k C .
\ee
\end{Proposition}

\begin{proof}
For this to make sense, we have to verify the $j$-composability of $(A,B\circ_k C)$ and the $k$-composability of $(A\circ_j B,C)$. Concerning the first, we have
\[
(B\circ_k C)s_j =(B\circ_k C)s_ms_j \stackrel{(\ref{comp2})}{=} Bs_j = At_j.
\]
The second works similarly.

We now prove~(\ref{associat}) by making use of~(\ref{2stepcomp}) and~(\ref{comp2}),
\begin{align*}
A\circ_j(B\circ_k C) & \stackrel{\eqref{2stepcomp}}{=} (A\circ_j (B\circ_k C)s_m)\circ_m (B\circ_k C) \\[5pt]
& \stackrel{\eqref{comp2}}{=} (A\circ_j B) \circ_m (B\circ_k C) \\[5pt]
& \stackrel{\eqref{comp2}}{=} (A\circ_j B)\circ_m((A\circ_j B)t_m\circ_k C) \stackrel{\eqref{2stepcomp}}{=} (A\circ_j B)\circ_k C.
\end{align*}
\end{proof}

\begin{Remark}
\label{augmented}
One can also work with \emph{augmented simplicial sets}, i.e.~presheaves $\C:\Delta^{\op}_+\to\Sets$ over the \emph{augmented simplex category} $\Delta_+$. Making the analogous definitions for composition, and possibly imposing that $\C(-1)$ be a singleton, one should obtain a definition of what might be coined an \emph{augmented compository}. In an augmented compository, \emph{all} pairs of morphisms $(A,B)\in\C(m)\times\C(n)$ are $(-1)$-composable, and their composition is an $(m+n+1)$-simplex $A\circ_{-1} B$.
Some of our upcoming examples can naturally be considered as augmented compositories. In this paper, though, we try to keep the number of newly introduced concepts somewhat limited, and hence we will not consider augmented compositories.
\end{Remark}

\section{Examples of compositories} 

Before studying in detail the examples mentioned in the introduction, we consider two examples of a category-theoretical nature: nerves of categories and higher spans in categories. These are two different ways of associating a compository to a category.

\subsection{Nerves of categories}
\label{nerves}

We recall the definition of the nerve of a category and then show how it naturally carries the structure of a compository.

\begin{Definition}[{\cite{Segal}}]
Given a small category $\sC$, the nerve of $\sC$ is the simplicial set $\mathcal{N}_\sC$ in which an $n$-simplex is a sequence of $n$ composable morphisms in $\sC$,
\[
\xymatrix{a_0\ar[r]^{f_1}&\ldots\ar[r]^{f_n}&a_n}.
\]
The action of the face map $\mathcal{N}_{\sC}(\partial_k)$ is given by, for $k\neq 0,n$,
\[\xymatrix{
\mathcal{N}_\sC(\partial_k)(a_0\ar[r]^(.65){f_1}&\ldots\ar[r]^(.35){f_n}&a_n)\defin(a_0\ar[r]^(.65){f_1}&\ldots\ar[r]^{f_{k-1}}& a_{k-1}\ar[r] \ar[r]^{f_{k+1}\circ f_k}& a_{k+1}\ar[r]^{f_{k+1}}&\ldots\ar[r]^{f_n}&a_n)},
\]
while for $k=0$,
\[
\xymatrix{
\mathcal{N}_\sC(\partial_k)(a_0\ar[r]^(.65){f_1}&\ldots\ar[r]^(.35){f_n}&a_n)\defin(a_1\ar[r]^(.65){f_2} &\ldots\ar[r]^{f_n}&a_{n})},
\]
and for $k=n$,
\[
\xymatrix{
\mathcal{N}_\sC(\partial_k)(a_0\ar[r]^(.65){f_1}&\ldots\ar[r]^(.35){f_n}&a_n)\defin(a_0\ar[r]^(.65){f_1} &\ldots\ar[r]^{f_{n-1}}&a_{n-1})}.
\]
The action of degeneracy maps is defined as 
\[
\xymatrix{
\mathcal{N}_\sC(\eta_k)(a_0\ar[r]^(.65){f_1}&\ldots\ar[r]^(.35){f_n}&a_n)\defin(a_0\ar[r]^(.65){f_1} & \ldots \ar[r]^{f_k} & a_k \ar@{=}[r]^{\id} & a_k\ar[r]^{f_{k+1}} & \ldots\ar[r]^{f_n} & a_n)}.
\]
\end{Definition}

Upon identifying a sequence of $n$ composable morphisms with a functor $[n]\to\sC$, one can also say that the presheaf $\mathcal{N}_\sC:\Delta^\op\to\Sets$ is the composition of functors
\be
\label{abstractnerve}
\xymatrix{ \Delta^\op\ar@{^{(}->}[r] & \Cats^\op\ar[rr]^{\Cats(-,\sC)} && \Sets },
\ee
where the first arrow is the functor which regards every finite ordinal $[n]$ as a category.

Turning $\mathcal{N}_\sC$ into a compository is almost trivial: we take the composition operation to be given by concatenation of paths. More concretely, a $k$-composable pair $(A,B)$ is a pair of simplices of the form
\begin{align}
\begin{split}
\xymatrix@C=1cm@R=.4cm{
A \ar@{}[r]|{=} & (a_0\ar[r]^{f_1} & \ldots \ar[r]^{f_{m-k}} & a_{m-k}\ar[r]^{f_{m-k+1}} & \ldots\ar[r]^{f_m} & a_m)\\
B \ar@{}[r]|{=} &&& (a_{m-k}\ar[r]^{f_{m-k+1}} & \ldots\ar[r]^{f_m} & a_m\ar[r]^{f_{m+1}} & \ldots\ar[r]^(.4){f_{m+n-k}} & a_{m+n-k})\,,}
\end{split}
\end{align}
and we can simply put
\be
\xymatrix@C=1cm{A\circ_k B \ar@{}[r]|{{\defin}} & (a_0\ar[r]^{f_1} & \ldots \ar[r]^(.45){f_{m-k}} & a_{m-k}\ar[r]^{f_{m-k+1}} & \ldots\ar[r]^{f_m} & a_m \ar[r]^{f_{m+1}} & \ldots\ar[r]^(.4){f_{m+n-k}} & a_{m+n-k})}. \hspace{1pc}
\ee
It is useful to understand this definition of $\circ_k$ more abstractly. Consider
\be
\label{deltapushout}
\vxymatrix{ [k]\ar[r]^{s_k}\ar[d]_{t_k} & [n]\ar[d]^{t_n} \\
[m] \ar[r]_(.35){s_m} & [m+n-k] }
\ee
as a pushout diagram in $\Cats$; applying its universal property to $A:[m]\to\sC$ and $B:[n]\to\sC$ is possible precisely when $(A,B)$ is $k$-composable, and the resulting functor $[m+n-k]\to\sC$ is the composition $A\circ_k B$.

\begin{Proposition}
With these definitions, $\mathcal{N}_\sC$ is a compository.
\end{Proposition}

\begin{proof}
The identity axiom~\eqref{comp} is immediate. 

The back-and-forth equation~\eqref{bf1} follows from the uniqueness part of the universal property in the diagram of pushouts
\[\xymatrix{
[k]\ar[r]^(.4){s_k}\ar[d]^{t_k}&[k+i]\ar[r]^(.46){s_{k+i}}\ar[d]^{t_{k+i}}&[n]\ar[d]^{t_n}\ar@/^1.5pc/[ddr]^{B}\\
[m]\ar[r]^(.4){s_m}\ar@/_1.8pc/[drrr]_{A}&[m+i]\ar[r]^(.42){s_{m+i}}\ar@/_.7pc/[drr]|(.4){(A\circ_kB)s_{m+i}}&[m+n-k]\ar[dr]|{A\circ_k B}\\
&&&\sC
}\]
and similarly for the other back-and-forth equation~\eqref{bf2}.

For composition with degeneracy maps in the form~\eqref{compdeg1}, we note that both squares in
\[
\xymatrix@C=2cm@R=1cm{ [k]\ar[d]^{t_k}\ar[r]^(.45){s_k} & [n]\ar[d]^{t_{n}}\ar@/^2.9pc/[dddr]^B &  \\
[m+1] \ar[r]^(.4){s_{m+1}}\ar[d]^{\eta_i} & [m+1+n-k]\ar[d]^{\eta_i}\ar@/^1pc/[ddr]|(.41){(A\circ_k B)\eta_i}|(.5){=}|(.59){A\eta_i\circ_k B} \\
[m]\ar@/_1.8pc/[drr]_A\ar[r]^(.45){s_m}&[m+n-k] \ar@/_.3pc/[dr]|{A\circ_k B} \\ 
&& \sC }
\]
are pushouts as well; that equation~\eqref{compdeg3} holds follows from an analogous diagram. Concerning~\eqref{compdeg2}, we observe that the front and the back square of the commutative cube

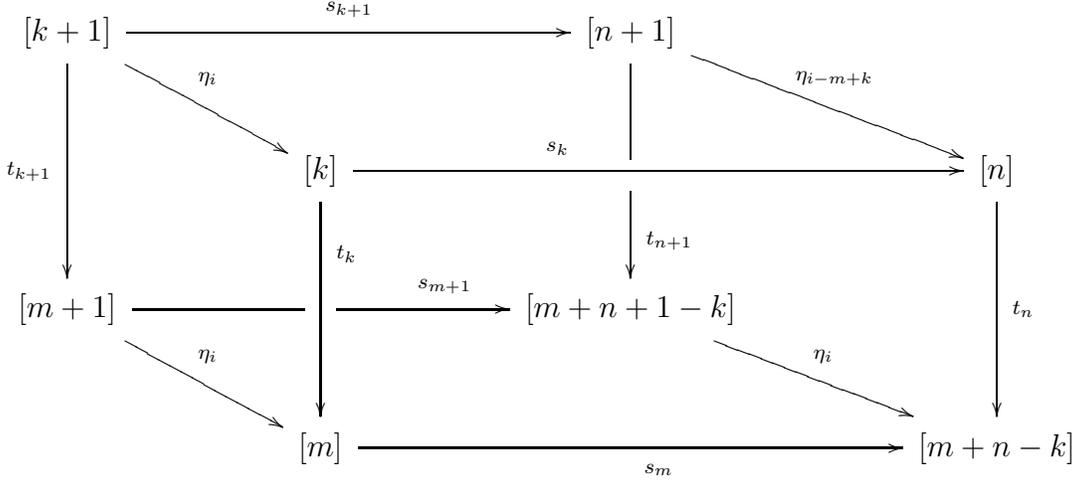
\[
\xymatrix@C=2cm@R=1cm{ [k+1]\ar[rr]^{s_{k+1}}\ar[dr]^{\eta_i}\ar[dd]_{t_{k+1}} && [n+1]\ar[dr]^{\eta_{i-m+k}}\ar'[d][dd]^{t_{n+1}} \\
& [k]\ar[rr]^(.35){s_k}\ar[dd]^(.3){t_k} && [n]\ar[dd]^{t_n} \\
[m+1] \ar[dr]^{\eta_i}\ar'[r][rr]^(.4){s_{m+1}} && [m+n+1-k]\ar[dr]^{\eta_{i}} \\
& [m]\ar[rr]_{s_m} && [m+n-k] }
\]
are pushouts, which also implies the desired conclusion by the universal property. The compatibility between composition and face maps can be proven in a very similar way.
\end{proof}

Notably, also property~\eqref{facenot} holds for $\mathcal{N}_\sC$. Indeed even as a simplicial set, $\mathcal{N}_\sC$ is very special: for every $k$-composable pair $(A,B)\in\mathcal{N}_\sC(m)\times\mathcal{N}_\sC(n)$, there is \emph{exactly one} simplex which has $A$ as its initial $m$-face and $B$ as its terminal $n$-face, namely the composition $A\circ_k B$. In other words, there is a unique structure of compository on the simplicial set $\mathcal{N}_\sC$. This is really just a restatement of the universal property of the pushout~\eqref{deltapushout}, which we have used repeatedly in the proof. It is also known as the \emph{Segal condition}, see~\cite{Segal} where it is attributed to Grothendieck.

We find it curious that the composition of the original category is not encoded in the composition of the resulting compository, but rather in its face maps.

Given that the structure of the nerve of a category is accurately captured by that of a compository, we now ask: are the nerves of higher categories~\cite{Street} also compositories? The nerves of strict $\omega$-categories have already been characterised as complicial sets~\cite{Verity} and recently as sets with complicial identities~\cite{Steiner}. In the latter characterisation, the \emph{wedge} operation also increases the dimension of simplices. However, the following argument---based on an idea by Richard Steiner\footnote{personal communication.}---shows that compositories cannot model the nerves of higher categories.

\begin{Theorem}[Steiner]
There is a strict $2$-category whose nerve cannot be equipped with a compository structure.
\label{nohighercats}
\end{Theorem}

\begin{proof}
Let $\sC$ be the strict $2$-category freely generated by the diagram
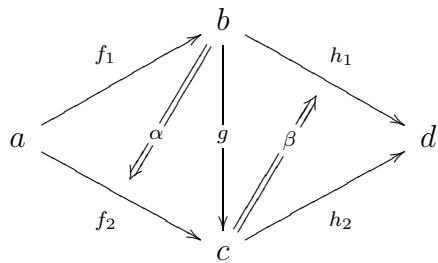
\[\xymatrix{
&& b \ar[dd]|*+<4pt>{\scriptstyle{g}} \ar[drr]^{h_1}_{}="2" \\
a\ar[urr]^{f_1} \ar[drr]_{f_2}^{}="1" \ar@{=>}[urr];"1"|(.65)*+<4pt>{\scriptstyle{{\alpha}}} &&&& d \\
&& c \ar[urr]_{h_2} \ar@{=>}[];"2"|(.65)*+<4pt>{\scriptstyle{{\beta}}} }
\]
Its nerve $\mathcal{N}_\sC$ contains a $2$-cell $S$ containing $\alpha$ and a $2$-cell $T$ containing $\beta$. These are $1$-composable since $\partial_0S=\partial_2T=(b\stackrel{g}{\to}c)$. However, there is \emph{no} $3$-cell $R$ with $\partial_3R=S$ and $\partial_0R=T$: such an $R$ would have to contain an additional $1$-cell from $A$ to $D$, together with $2$-cells to this additional $1$-cell from $h_1f_1$ and from $h_2f_2$, respectively. Checking all three possibilities for this additional $1$-cell shows that none of them has the required property.
\end{proof}

\begin{Remark}
But curiously enough, there is a new approach to semistrict higher categories~\cite{kv} in which an $n$-cell and an $m$-cell compose along a common $k$-cell to an $(m + n - k - 1)$-cell\footnote{This has been pointed out to us by an anonymous referee.}, which almost matches the cell dimensions of composition in a compository. However, we are not aware of a connection.
\end{Remark}

\subsection{Higher spans}
\label{higherspans}

Spans and higher spans come up in the study of topological quantum field theories, and various categorical structures have been proposed for talking about them~\cite{grandis,haugseng}. Here, we explain how higher spans in many categories can be regarded as forming a compository.

In the following, let $\sC$ be any category with pullbacks in which all isomorphisms are identities (gaunt category \cite{clark}). The treatment of an arbitrary category $\sC$ will be discussed at the end of this subsection. The reason we consider gaunt categories is because this property guarantees the strict uniqueness of all limits and Kan extensions of functors with codomain $\sC$.

We now move to considering spans in $\sC$, which are diagrams of the form
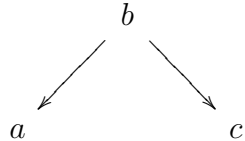
\[
\xymatrix{ & b\ar[dr]\ar[dl] \\ a && c }
\]
Usually, spans are regarded as morphisms in a bicategory of spans~\cite{Benabou}. A pair of composable spans takes the form
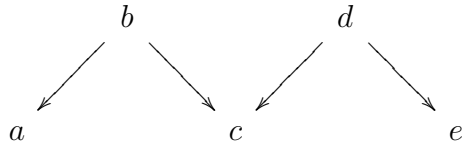
\[
\xymatrix{ & b\ar[dl]\ar[dr] && d\ar[dl]\ar[dr] \\ a && c && e }
\]
and their composition in the bicategory of spans is usually defined to be the span arising from the diagram
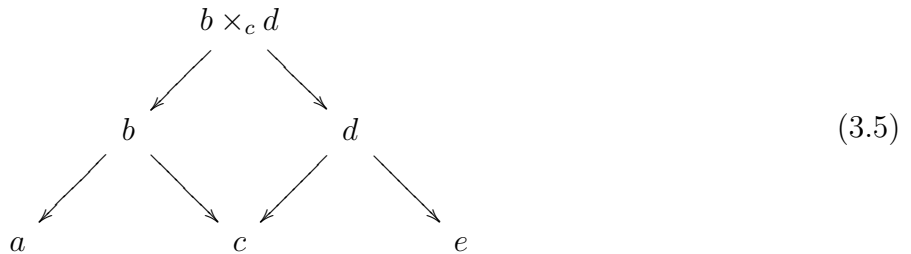
\be
\begin{split}
\label{spancomp1}
\xymatrix@!@=0cm{ && b\times_c d\ar[dl]\ar[dr] \\ & b\ar[dl]\ar[dr] && d\ar[dl]\ar[dr] \\ a && c && e }
\end{split}
\ee
upon composing the two outer legs in $\sC$ and forgetting the two arrows arriving at $c$. 

Given the previous considerations, it seems natural to try to do without this ``forgetting'' operation and retain the whole diagram~(\ref{spancomp1}) as a $2$-simplex in a compository. And indeed, in this way we obtain a compository $\mathcal{S}_\sC$ of higher spans in $\sC$. We now embark on the details of this.

\begin{Definition}
For $n\in\Nl$, the \emph{walking $n$-span} $\Sp_n$ is the poset with objects $(v,w)\in\N$ with $v\leq w\leq n$ and
\be
\label{walkingspandef}
(v,w)\leq (v',w') \quad\Longleftrightarrow\quad v\leq v'\:\textrm{ and } \: w'\leq w.
\ee
\end{Definition}

We think of an object $(v,w)\in\Sp_n$ as an \emph{interval} in the poset $[n]$, and these intervals are ordered by reverse containment; in other words, $\Sp_n$ is the \emph{interval domain}~\cite{Scott} associated to the poset $[n]$. In category-theoretic terms, this means that $\Sp_n$ is the twisted arrow category~\cite{MacLane} associated to $[n]=\{0,\ldots,n\}$: an object $(v,w)\in\Sp_n$ can be identified with the arrow $v\stackrel{\leq}{\longrightarrow} w$ in $[n]$, while an arrow $(v,w)\stackrel{\leq}{\longrightarrow}(v',w')$ corresponds to a diagram
\be
\vxymatrix { v\ar[r]^{\leq} \ar[d]_{\leq} & w\ar@{<-}[d]^{\leq} \\
 v'\ar[r]^{\leq} & w' }
\ee
For example, the walking $2$-span $\Sp_2$ is the category generated by the directed graph
\be
\begin{split}
\xymatrix@!@=.5cm{ && (0,2)\ar[dl]\ar[dr] \\ & (0,1)\ar[dl]\ar[dr] && (1,2)\ar[dl]\ar[dr] \\ (0,0) && (1,1) && (2,2) }
\end{split}
\ee
such that the square commutes.

Since taking the twisted arrow category is a functor $-^\mathrm{tw}:\Cats\to\Cats$, composing with the inclusion $\Delta\hookrightarrow\Cats$ gives a functor
\[
\xymatrix{ \Delta\ar@{^{(}->}[r] & \Cats\ar[r]^{-^\mathrm{tw}} & \Cats }
\]
which takes $[n]$ to $\Sp_n$. Then in analogy with~\eqref{abstractnerve}, the composition of functors
\be\label{Ssimplicial}
\vxymatrix{ \Delta^{\op}\ar@{^{(}->}[r] & \Cats\ar[r]^{-^\mathrm{tw}} & \Cats\ar[rr]^{\Cats(-,\sC)} && \Sets }
\ee
defines a simplicial set $\mathcal{S}_\sC:\Delta^{\op}\to\Sets$. More concretely, we can equivalently define its $n$-simplices as $n$-spans in $\sC$:

\begin{Definition}
An \emph{$n$-span} in $\sC$ is a functor $\Sp_n\to\sC$.
\end{Definition}

It is straightforward to show that the resulting face maps $\mathcal{S}_\sC(\partial_k)$ are given by composition with the functors
\[
\Sp_{n-1}\longrightarrow \Sp_n,\qquad (v,w)\mapsto \mleft(\partial_k(v),\partial_k(w)\mright) =  \begin{cases} (v,w) & \text{ if } w<k \\ (v,w+1) & \text{ if } v< k\leq w\\ (v+1,w+1) & \text{ if } k\leq v\end{cases}.
\]
By definition, this is the inclusion functor which misses the two ``lines'' of objects
\[
(0,k),\ldots,(k,k) \quad\text{ and }\quad (k,k),\ldots,(k,n).
\]
Similarly, the degeneracies $\mathcal{S}_\sC(\eta_k)$ arise from the functors
\[
\Sp_{n+1}\longrightarrow \Sp_n,\qquad (v,w)\mapsto \mleft(\eta_k(v),\eta_k(w)\mright) = \begin{cases} (v,w) & \text{ if } w\leq k \\ (v,w-1) & \text{ if } v\leq k< w\\ (v-1,w-1) & \text{ if } k<v \end{cases}.
\]

We now turn to composition of these higher spans. A pair $(A,B)$ with $A:\Sp_m\to\sC$ and $B:\Sp_n\to\sC$ is $k$-composable if the terminal $k$-face of $A$ coincides with the initial $k$-face of $B$. Equivalently, the functors $A$ and $B$ assemble into a functor
\[
[A, B]_{\Sp_k}\::\:\Sp_m \textstyle\coprod_{\Sp_k} \Sp_n\longrightarrow \sC.
\]
where $\Sp_k$ is included in $\Sp_m$ and $\Sp_n$ via $t_k=\partial_0\cdots\partial_0$ and $s_k=\partial_n\cdots\partial_{k+1}$, respectively.

We think of $\Sp_m \textstyle\coprod_{\Sp_k}\Sp_n$ as the ``walking $k$-composable pair'' consisting of an $m$-span and an $n$-span sharing a common $k$-span. Upon regarding $\Sp_m$ as the initial $m$-face of $\Sp_{m+n-k}$ and $\Sp_n$ as the corresponding terminal $n$-face, we obtain a functor $c_{m,k,n}:\Sp_m \textstyle\coprod_{\Sp_k}\Sp_n\longrightarrow \Sp_{m+n-k}$.

\begin{Definition}
The composition $A\circ_k B:\Sp_{m+n-k}\rightarrow \sC$ is the right Kan extension of $[A, B]_{\Sp_k}$ along $c_{m,k,n}$. 
\end{Definition}

Due to the assumption that $\sC$ is gaunt, this Kan extension is necessarily unique. By the pointwise construction of Kan extensions~\cite[Thm.~X.3.1]{MacLane} and the particular form of the categories involved, it can be computed in terms of pullbacks. Explicitly, $A\circ_k B:\Sp_{m+n-k}\rightarrow \sC$
 is the higher span with objects 
\[
(A\circ_k B)(v,w) = \begin{cases} A(v,w) & \text{if } w\leq m\\ A(v,m)\times_{B(0,k)}B(0,w-m+k) & \text{if } v<m-k \text{ and }\: m<w \\ B(v-m+k,w-m+k) & \text{if } m-k\leq v \end{cases}
\]
and the obvious morphisms. The first and third cases are not disjoint; the compatibility assumption $At_k=Bs_k$ guarantees that the left-hand side is nevertheless well-defined. It also guarantees that $B(0,k)=A(m-k,m)$, hence there is no asymmetry between $A$ and $B$ in the second case.

For example for $k=0$, the composition of a $2$-span $A$ with a $1$-span $B$ is given by the $3$-span
\[
\begin{split}
\xymatrix@!@=-2cm{ &&& A(0,2)\times_{B(0,0)} B(0,1)\ar@{-->}[dl]\ar@{-->}[dr] \\ && A(0,2)\ar[dl]\ar[dr] && A(1,2)\times_{B(0,0)}B(0,1)\ar@{-->}[dl]\ar@{-->}[dr] \\ & A(0,1)\ar[dl]\ar[dr] && A(1,2)\ar[dl]\ar[dr] && B(0,1)\ar[dl]\ar[dr] \\ A(0,0) && A(1,1) && A(2,2)=B(0,0) && B(1,1) }
\end{split}
\]

\begin{Theorem}\label{Scompository}
With these definitions, $\mathcal{S}_\sC$ becomes a compository.
\end{Theorem}

\begin{figure}
\begin{center}
\begin{tikzpicture}
\draw (0,0)--(3,3) -- (6,0) ;
\draw (3,0) -- (6,3) -- (9,0) ;
\draw (5,0) -- (9,4) -- (13,0) ;
\draw[dashed] (3,3) -- (4.5,4.5) -- (6,3) -- (8,5) -- (9,4) ;
\draw[dotted] (4.5,4.5) -- (6.5,6.5) -- (8,5) ;
\node at (3,2) {$A$} ;
\node at (6,2) {$B$} ;
\node at (9,2) {$C$} ;
\node at (4.5,3) {$A\circ_j B$} ;
\node at (7.5,3.5) {$B\circ_k C$} ;
\node at (6.5,5) {$A\circ_j B\circ_k C$} ;
\end{tikzpicture}
\end{center}
\caption{Schematic illustration of three composable higher spans and their compositions.}
\label{schematicspans}
\end{figure}
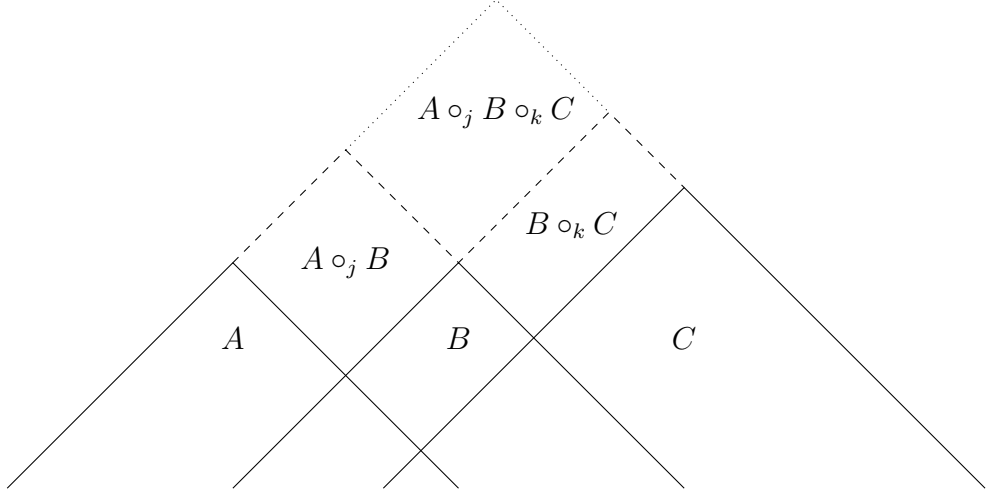

\begin{proof}
The identity axiom is trivially satisfied. Moreover,~\cite[Cor.~X.3.4]{MacLane} shows that $(A\circ_k B)c_{m,k,n}=[A,B]_{\Sp_k}$, so that 
$(A\circ_k B)s_m=A$ and $(A\circ_k B)t_n=B$. 

For the back-and-forth axiom we need to show that $(A\circ_k B)s_{m+i}\circ_{k+i} B=A\circ_k B$; that this is also equal to $A\circ_{k+j} (A\circ_k B)t_{n+j}$ can be shown in a similar way.

Consider the diagram 
\[\vxymatrix{
\Sp_m\coprod_{\Sp_k}\Sp_n\ar@/^3pc/@{^{(}->}[rrdd]^{JI}\ar[dd]_{[A,B]_{\Sp_k}}\ar@{^{(}->}[dr]^{I}\\
&\Sp_{m+i}\coprod_{\Sp_{k+i}}\Sp_n\ar@{^{(}->}[dr]^J\ar[dl]|{F{\defin}\mathrm{Ran}_I[A,B]} \\
\sC&&\Sp_{m+n-k}\ar[ll]^{G{\defin}\mathrm{Ran}_JF} 
}\qquad
\vxymatrix{ I\defin s_m^{\mathrm{tw}}\coprod_{s_k^{\mathrm{tw}}}\id \\  J\defin c_{m+i,k+i,n} }
\]
The definition of $G$ implies that it is also the right Kan extension of $[A,B]_{\Sp_k}$ along $JI$, i.e.~$G=A\circ_k B$. By~\cite[Cor.~X.3.4]{MacLane}, $F=GJ$, so that $F=[(A\circ_k B)s_{m+i}, B]_{\Sp_k}$. By definition of composition by Kan extension, $G=(A\circ_k B)s_{m+i}\circ_{k+i} B$. 

To prove compatibility with both face maps and degeneracy maps, we will  work with a general monotone map $\xi:[m'+n'-k']\rightarrow [m+n-k]$ which satisfies the following conditions:
\begin{enumerate}
\item $v'\leq m'\Longrightarrow \xi(v')\leq m$
\item $v'\geq m'-k'\Longrightarrow \xi(v')\geq m-k$
\item $m-k\leq v\leq m\Longrightarrow\;\exists v' \text{ with } m'-k'\leq v'\leq m'\; \text{such that }\xi(v')=v$
\end{enumerate}
These imply in particular
\be
\label{boundaries}
\xi(m'-k')=m-k,\qquad \xi(m')=m.
\ee
Upon specialising $\xi$ to the appropriate face and degeneracy maps, all five axioms~\eqref{compdeg1}--\eqref{faceT} become special cases of the following argument.

We consider the square  
\be\label{square}\vxymatrix{
\Sp_{m'}\coprod_{\Sp_{k'}}\Sp_{n'}\ar[rr]^{c_{m',k',n'}}\ar[d]^{\xi|_{}}&&\Sp_{m'+n'-k'}\ar[d]^{\xi}\\
\Sp_m\coprod_{\Sp_k}\Sp_n\ar[rr]^{c_{m,k,n}}&&\Sp_{m+n-k}
}\ee
where $\xi_|$ is the restriction of $\xi$. It needs to be shown that the right Kan extension of $[A,B]_{\Sp_k}\xi|_{}$ along $c_{m',k',n'}$ can be computed as 
$$
\mathrm{Ran}_{c_{m',k',n'}}([A,B]_{\Sp_k}\xi|_{})=(\mathrm{Ran}_{c_{m,k,n}}[A,B]_{\Sp_k})\xi
$$
In modern terminology, this means that we need to show that~\eqref{square} is an \emph{exact square}~\cite{carres,exsqnlab}.
Using the combinatorial characterisation of exact squares~\cite{carres,exsqnlab} and the fact that all four categories in~\eqref{square} are posets, this boils to proving that for any $(v,w)\in\Sp_m\coprod_{\Sp_k}\Sp_n$ and $(v',w')\in\Sp_{m'+n'-k'}$ with $\xi(v',w')\rightarrow (v,w)$, the poset 
$$
\left\{(x',y')\in \Sp_{m'}\coprod_{\Sp_{k'}}\Sp_{n'}\;\Biggm|\;(v',w')\rightarrow(x',y'),\;\xi|_{}(x',y')\rightarrow (v,w)\right\}
$$
is non-empty and connected. Spelling out these conditions gives that the assumptions
$$
\begin{matrix}  v\leq w\\[5pt]v'\leq w'\\[5pt] m-k\leq v\;\lor\; w\leq m\\[5pt] \xi(v')\leq v,\quad w\leq \xi(w')\end{matrix}$$
should imply that
\[
	I\defin\left\{(x',y')\in[m'+n'-k']^{\times 2} \;\;\Biggm|\;\; \begin{matrix} m'-k'\leq x' \,\lor\, y'\leq m'\\[5pt] v'\leq x',\quad y'\leq w'\\[5pt] w\leq \xi(y'),\quad \xi(x')\leq v\end{matrix}\;\right\}
\]
is non-empty and connected with respect to the ordering induced from $\Sp_{m'}\coprod_{\Sp_{k'}}\Sp_{n'}$. If $v'\geq m'-k'$ or $w'\leq m'$, then $(v',w')\in I$ is a least element. In particular, $ I$ is non-empty and connected. Hence for the remainder of this proof, we can assume $v'<m'-k'$ and $w'>m'$, and distinguish three cases:
\begin{enumerate}[leftmargin=6cm]
\item[\underline{Case $v<m-k$ and $w\leq m$:}] We claim that $(v',m')$ is a least element of $ I$. Indeed $(v',m')\in I$ since $m'<w'$ and $w\leq m=\xi(m')$ and $\xi(v')\leq v$ as assumed above. For any $(x',y')\in I$, we have $v'\leq x'$ and $y'\leq m'$ since the case $m'-k'\leq x'$ is impossible due to $\xi(x')\leq v<m-k=\xi(m'-k')$, and hence $(v',m')\to(x',y')$, as claimed. So since $ I$ has a least element, it is non-empty and connected. 
\item[\underline{Case $v\geq m-k$ and $w>m$:}] This is analogous to the previous case, where the least element now is $(m'-k',w')\in I$.
\item[\underline{Case $v\geq m-k$ and $w\leq m$:}] Due to the same reasoning as in the previous two cases, we have $(v',m')\in I$ and $(m'-k',w')\in I$, which implies that $ I$ is non-empty. For any other $(x',y')\in I$, we have $(v',m')\to(x',y')$ if $y'\leq m'$, while $(m'-k',w')\to(x',y')$ if $m'-k'\leq x'$. On the other hand, we have $(m'-k',m')\in I$ with $(v',m')\to(m'-k',m')$ and $(m'-k',w')\to(m'-k',m')$. This shows connectedness.\qedhere
\end{enumerate}
\end{proof}

\begin{Example}\label{nokan}
In general, $\mathcal{S}_{\sC}$ is not a weak Kan complex (quasi-category~\cite{Joyal}). For example for the category $\sC=[2]\times[2]$, which is a poset with Hasse diagram 
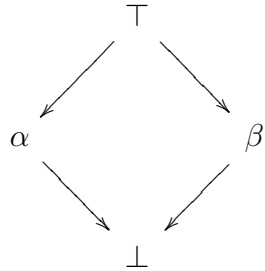
\[\xymatrix{
&\top\ar[dr]\ar[dl]\\
\alpha\ar[dr]&&\beta\ar[dl]\\
&\bot
}
\]
this can be seen as follows. The three $2$-spans 
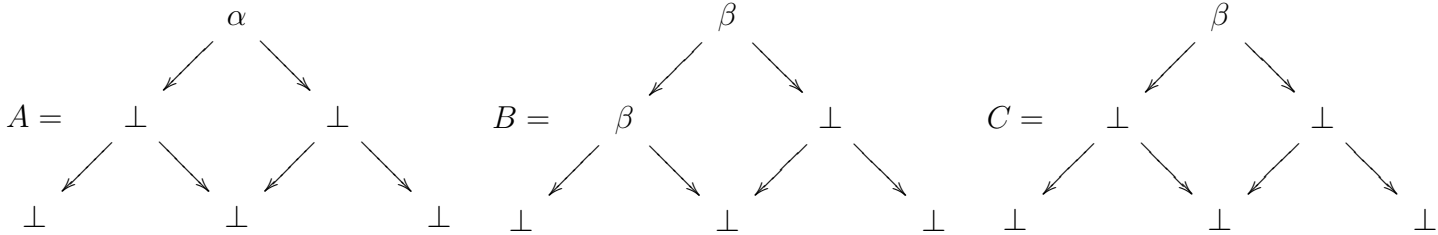
\[\xymatrix@-=5pt@!{
&&\alpha\ar[dr]\ar[dl]\\
A=&\bot\ar[dr]\ar[dl]&&\bot\ar[dr]\ar[dl]\\
\bot&&\bot&&\bot
}\;\xymatrix@-=5pt@!{&&\beta\ar[dr]\ar[dl]\\
B=&\beta\ar[dr]\ar[dl]&&\bot\ar[dr]\ar[dl]\\
\bot&&\bot&&\bot}\;\xymatrix@-=5pt@!{&&\beta\ar[dr]\ar[dl]\\
C=&\bot\ar[dr]\ar[dl]&&\bot\ar[dr]\ar[dl]\\
\bot&&\bot&&\bot}\] 
satisfy
\[
	A\partial_0=B\partial_0,\qquad B\partial_2=C\partial_1,\qquad C\partial_0=A\partial_2
\]
and hence assemble into an inner horn $\Lambda^3_2\rightarrow \mathcal{S}_{\sC}$ as illustrated by 
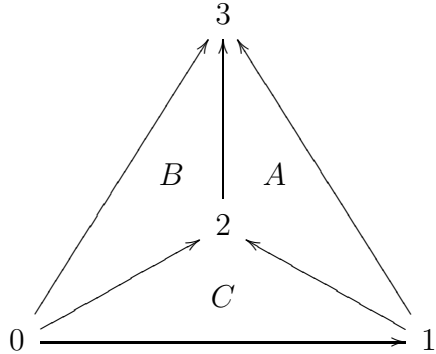
\[
\xymatrix{ && 3 \\ & &&\\&& 2\ar[uu]\ar@{}[d]|(.6){\textstyle{C}}\ar@{}[ul]|{\textstyle{B}}\ar@{}[ur]|{\textstyle{A}} \\ 0\ar[uuurr] \ar[urr]\ar[rrrr] &&&& 1\ar[uuull]\ar[ull]}
\]
 Any potential $3$-span filler $D$ of this inner horn needs to satisfy $D\partial_0=A$, which implies $D(1,2)=A(0,2)=\alpha$, and $D\partial_1=B$, which implies $D(0,3)=B(0,2)=\beta$. Since $\beta\not\to\alpha$, such a $D$ does not exist, and hence $\mathcal{S}_\sC$ is not a quasi-category.
 
Since this argument did not actually make use of $C$, there is not even any $3$-span having $A$ as its $0$th face and $B$ as its $1$st face.
\end{Example}

The nerve $\mathcal{N}_\sC$ is a subcompository of $\mathcal{S}_\sC$ as follows. For every $n\in\Nl$, there are two monotone maps
$$
\textphnc{p},\textphnc{\ARp}\;:\;\Sp_n\longrightarrow [n],\qquad\textphnc{p}(v, w)\defin v,\qquad\textphnc{\ARp}(v,w)\defin n-w.
$$
Precomposition with either \textphnc{p} or \textphnc{\ARp} turns a functor $[n]\rightarrow \sC$ into a functor $\Sp_n\rightarrow \sC$. We regard this as a map $\mathcal{N}_\sC(n)\rightarrow \mathcal{S}_\sC(n)$. Since both functors have left inverses, this exhibits $\mathcal{N}_{\sC}(n)$ as a subset of $\mathcal{S}_\sC(n)$ in two ways. Our goal is to show that this inclusion respects the compository structure:

\begin{Proposition}
With these definitions, $\mathcal{N}_\sC$ is a subcompository of $\mathcal{S}_\sC$ in two ways.
\end{Proposition}

\begin{proof}
We will give the proof only for $\textphnc{p}$; an analogous proof applies to \textphnc{\ARp}.
All that needs to be shown is compatibility with face/degeneracy maps and composition. The first follows from the commutativity of 
\[\xymatrix{
\Sp_n\ar[r]^{\xi^{\mathrm{tw}}}\ar[d]_{\textphnc{p}}&\Sp_{n'}\ar[d]^{\textphnc{p}}\\
[n]\ar[r]^{\xi}&[n']
}
\]
for any $\xi\in\Delta([n],[n'])$; this square is an instance of naturality of $\textphnc{p}$, regarded as the restriction of a natural transformation $-^{\mathrm{tw}}\mathrel{\to}\id_{\mathrm{Cats}}$.

Compatibility with composition is similar to the second half of the proof of Theorem~\ref{Scompository}. In concrete terms, we need to show that the square
\[\xymatrix{
\Sp_m\coprod_{\Sp_k}\Sp_n\ar[r]^-{c_{m,k,n}}\ar[d]_{\textphnc{p}\coprod\textphnc{p}}&\Sp_{m+n-k}\ar[d]^{\textphnc{p}}\\
[m]\coprod_{[k]}[n]\ar@{=}[r]&[m+n-k]
}\]
is exact. This means that for every $(v,w)\in \Sp_{m+n-k}$ and $z\in [m+n-k]$ with $v\leq z$, the poset 
$$
\left\{(x,y)\in\Sp_m\amalg_{\Sp_k}\Sp_n\;\bigg{|}\;(v,w)\rightarrow (x,y),\;(\textphnc{p}\amalg\textphnc{p})(x,y)\rightarrow z\right\}
$$
is non-empty and connected. As in the proof of Theorem~\ref{Scompository}, a case distinction together with the consideration of the least elements shows that this is indeed true.
\end{proof}

So far, we have defined the compository of higher spans $\mathcal{S}_\sC$ only when the original category $\sC$ is gaunt. Unfortunately, any attempt at a general definition soon runs into coherence issues: uniqueness of the Kan extensions used in the compositions is lost, and the question is whether they can be chosen coherently in such a way that the compository axioms hold with equality. 

One way to achieve this may be to choose pullbacks in $\sC$ such that $a\times_b b=a$ for any diagram of the form
\be
\vxymatrix{ & b\ar@{=}[d] \\
 a\ar[r] & b }
\ee
as well as similarly $a\times_a b=b$, and $a\times_b (b\times_c d) = a\times_c d$ for any diagram of the form
\be
\vxymatrix{ && d\ar[d] \\
 a\ar[r] & b\ar[r] & c }
\ee
and similarly upon ``extending'' the vertical leg instead of the horizontal one. In particular, this implies that for any diagram of the form
\be
\vxymatrix{ a \ar[rd] && c\ar[ld]\ar[rd] && e\ar[ld] \\ & b && d }
\ee
the pullbacks strictly associate in the sense that $a\times_b (c \times_d e) = (a\times_b c)\times_c (c\times_d e) = (a\times_b c)\times_d e$.

Alternatively, one may be inclined to say that the equational axioms for compositories should only be postulated in a certain weak form, such that composition satisfies these equations only ``up to'' higher isomorphisms satisfying their own laws up to isomorphisms etc. While this certainly bears some truth, it also seems conceivable that this is but an artefact of the translation from categories to compositories, and that there is no reason from within the theory of compositories itself to weaken strict equations.

\subsection{Metric spaces}
\label{metspacesI}

We now return to the metric spaces example considered in the introduction and show how the collection of all finite metric spaces forms a compository, although with a slightly non-standard notion of ``metric''. 

Since we are not interested in distinguishing different but isomorphic metric spaces, we take the underlying set of any $(n+1)$-element metric space to be the abstract $n$-simplex $[n]=\left\{0,\ldots,n\right\}$.

For us, a \emph{metric} on $[n]$ is a function
$$
d:[n]\times[n]\longrightarrow\R_{\geq 0}
$$
such that $d(x,x)=0$ for all $x\in[n]$, the triangle inequality
$$
d(x,z)\leq d(x,y) + d(y,z)
$$
holds, and also symmetry $d(x,y)=d(y,x)$. So in contrast to the standard definition, our metrics are not required to be non-degenerate: they are \emph{pseudometrics}. Moreover, everything that follows also works without the symmetry assumption, and the so inclined reader~\cite{Lawv} may safely take our notion of ``metric space'' to mean ``category enriched over the additive monoid $\R_{\geq 0}$''.

\begin{Definition}
$\mathcal{M}_1$ is the simplicial set with $n$-simplices given by the metrics on $[n]$,
\be
\mathcal{M}_1(n)\defin \left\{\: d:[n]\times [n] \longrightarrow \R_{\geq 0} \:\:\textrm{metric on }[n]\:\right\}.
\ee
For every $f\in\Delta([n],[m])$ and $A\in\mathcal{M}_1(m)$, we put
\be
\label{metricpull}
(Af)(x,y)\defin A(f(x),f(y)) .
\ee
\end{Definition}

We think of an $m$-simplex $A\in\mathcal{M}_1(m)$ as a metric space with $(m+1)$ points. Unfolding the definition then shows that the face maps correspond to all possible restrictions of the metric to an $m$-element subset, while the degeneracy maps correspond to all possible ways of duplicating a point.

A pair $(A,B)\in\mathcal{M}_1(m)\times\mathcal{M}_1(n)$ is $k$-composable if the restrictions of $A$ and $B$ to the corresponding $k$-element subsets coincide. In this case, we put:

\begin{Definition} For all $x,z\in[m+n-k]$, 
\[
(A\circ_k B)(x,z) \defin \begin{cases} A(x,z)&\text{ if }x\leq m \text{ and }z\leq m\\[3pt]
\min\limits_{y\;:\;m-k\leq y\leq m} \mleft( A(x,y) + B(y-m+k,z-m+k) \mright)& \text{ if } x\leq m \text{ and } z\geq m-k\\
B(x,z)&\text{ if } x\geq m-k\text{ and } z\geq m-k\end{cases}
\]
\end{Definition}

Although these three cases overlap, the resulting values for the left-hand side coincide thanks to the assumption of composability and the triangle inequality.

This definition implements the idea that $A\circ_k B$ represents a canonical way of joining $A$ and $B$ into a larger metric space in the sense  that the given metrics $A$ and $B$ are retained, while the ``missing'' distances are lengths of shortest paths as illustrated in Figure~\ref{shortpath}.

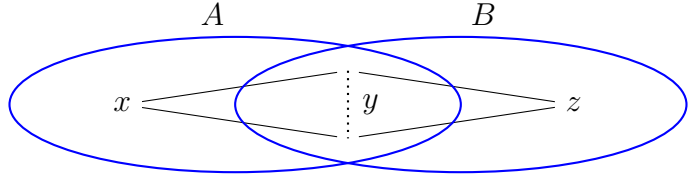
\begin{figure}\label{shortpath}
\begin{center}
\begin{tikzpicture}[scale=1.5]
\node[] (b) at (-2,0) {$x$} ;
\node[] (d) at (2,0) {$z$} ;
\node[] (f) at (0,.3) {} ;
\node[] () at (.2,0){$y$};
\node[] (a) at (0,-.3) {} ;
\draw [dotted,thick] (0,.3)--(0,-.3);
\draw (b) -- (f) -- (d) ;
\draw (b) -- (a) -- (d) ;
\node[] (c) at (-1.2,.8) {$A$} ;
\node[] (e) at (1.2,.8) {$B$} ;
\draw[thick,blue] (-1,0) ellipse (2cm and .6cm) ;
\draw[thick,blue] (1,0) ellipse (2cm and .6cm) ;
\end{tikzpicture}
\caption{Defining the distance between $x$ and $z$ in terms of a shortest path through points in the intersection.}
\end{center}
\end{figure}

\begin{Proposition}\label{metriccomp}
With these definitions, $\mathcal{M}_1$ is a compository. 
\end{Proposition}

The proof is straightforward but tedious. We will present a reasonably clean argument in Section~\ref{takeII}, once we have introduced notions which allow for a more direct proof.

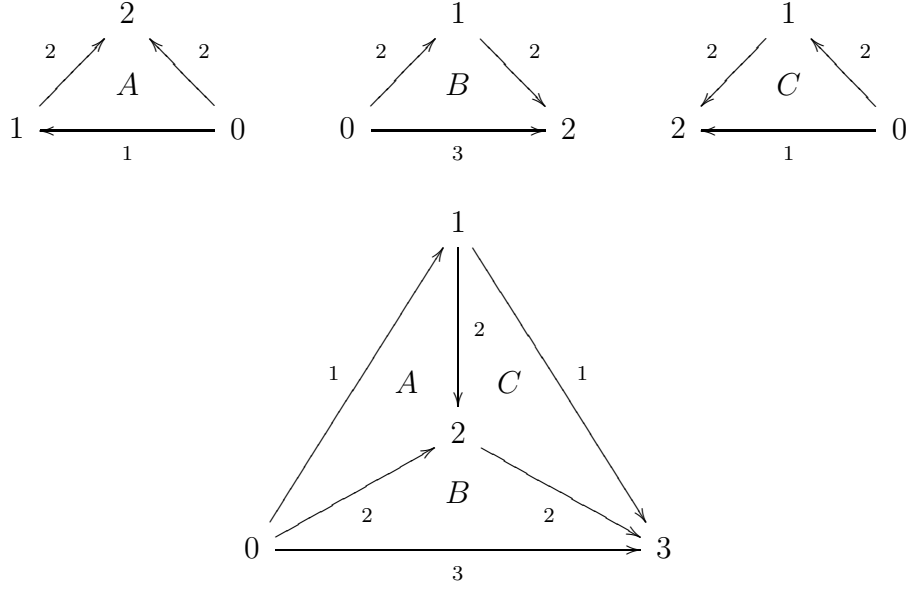
\begin{figure}
\[
 \xymatrix{ & 2\ar@{<-}[rd]^2\ar@{}[d]|(.6){\textstyle{A}} \\ 1 \ar[ur]^2 \ar@{<-}[rr]_1 && 0}\qquad
\xymatrix{ & 1\ar[rd]^2\ar@{}[d]|(.6){\textstyle{B}} \\ 0 \ar[ur]^2 \ar[rr]_3 && 2} \qquad \xymatrix{ & 1\ar@{<-}[rd]^2 \ar@{}[d]|(.6){\textstyle{C}}\\ 2 \ar@{<-}[ur]^2 \ar@{<-}[rr]_1 && 0} \]
\[
\xymatrix{ && 1\ar[dddrr]^1\ar[dd]^2 &&  \\&&&\\ && 2\ar[drr]_2\ar@{}[ul]|{\textstyle{A}}\ar@{}[ur]|{\textstyle{C}}\ar@{}[d]|{\textstyle{B}} && \\ 0\ar[uuurr]^1 \ar[urr]_2\ar[rrrr]_3 &&&& 3}
\]
\caption{Three metrics $A,B,C\in\mathcal{M}_1(2)$ with distances as indicated which assemble to an inner horn. Due to failure of the triangle inequality for the outer edges, this inner horn does not have a filler.}
\label{nofillers}
\end{figure}

\begin{Remark}
Similarly as in Example~\ref{nokan}, $\mathcal{M}_1$ is not a quasi-category: the three $2$-simplices $A,B,C\in\mathcal{M}_1(2)$ displayed in Figure~\ref{nofillers} form an inner horn in $\mathcal{M}_1$ which does not have a filler.
\end{Remark}

We will get back to the example of metric spaces in Section~\ref{takeII}, using the presheaf point of view discussed in the introduction.

\subsection{Joint probability distributions}
\label{secjpd}

We now consider the situation of joint probability distributions mentioned in the introduction. As we did there, we fix a finite set $O$ of outcomes for all our random variables. We will show that the collection of joint distributions of random variables with outcomes in $O$ form a compository denoted $\mathcal{P}_O$.

\begin{Definition}
An $n$-simplex in $\mathcal{P}_O$ is a probability distribution $P$ on $O^{n+1}=O\times\ldots\times O$ which assigns a weight $P(a_0,\ldots,a_n)$ to any $(n+1)$-tuple of outcomes $(a_0,\ldots,a_n)\in O^{n+1}$.
\end{Definition}

Hence an $n$-simplex is a joint probability distribution for $n+1$ random variables. In particular, a $0$-simplex is a probability distribution of a single variable.

The face maps are given by taking marginal distributions,
\be
(P\,\partial_k)(a_0,\ldots,a_{n-1}) \defin \sum_{a_k} P(a_0,\ldots,a_n),
\ee
while the degeneracies produce a ``copy'' of one of the variables which is perfectly correlated with the original one,
\be
(P\,\eta_k)(a_0,\ldots,a_{n+1}) \defin \delta_{a_k,a_{k+1}} \, P(a_0,\ldots,\cancel{a_{k+1}},\ldots,a_n).
\ee
These two equations can be subsumed into a single equation analogous to~\eqref{metricpull}: for any $f\in \Delta([m], [n])$, 
$$
(Pf)(a_0,\ldots,a_{m})\defin \sum_{b_0,\ldots,b_n\text{ s.t. } a_i=b_{f(i)}}P(b_0,\ldots,b_n).
$$
A simple calculation shows that this is functorial in $f$. 

Thus, a pair $(P, Q)$ consisting of an $m$-simplex $P$ and an $n$-simplex $Q$ is $k$-composable if and only if
\be
\label{probcomposable}
 \sum_{a_0,\ldots,a_{m-k-1}} P(a_0,\ldots,a_m) = \sum_{a_{m+1},\ldots,a_{m+n-k}} Q(a_{m-k},\ldots,a_{m+n-k})\qquad\forall a_{m-k},\ldots, a_m.
\ee
If this composability holds, then we abbreviate both sides of this equation by $R(a_{m-k},\ldots,a_m)$ and define
\be
\label{probcompose}
(P\circ_k Q)(a_0,\ldots,a_{m+n-k}) \defin \frac{P(a_0,\ldots,a_m)Q(a_{m-k},\ldots,a_{m+n-k})}{R(a_{m-k},\ldots,a_m)}.
\ee
In this formula, the denominator may vanish for certain tuples $(a_{m-k},\ldots,a_m)\in O^k$; however, the composability condition~\eqref{probcomposable} guarantees that this implies that also both terms in the numerator vanish. In this case, also the left-hand side of~\eqref{probcompose} is regarded to be $0$.

We need to check that \eqref{probcompose} is indeed a probability distribution. While it is clearly non-negative, normalisation can be seen as follows:
\begin{align*}
\sum_{a_0,\ldots,a_{m+n-k}}\frac{P(a_0,\ldots,a_m)Q(a_{m-k},\ldots,a_{m+n-k})}{R(a_{m-k},\ldots,a_m)}&=\sum_{a_{0},\ldots, a_m}\frac{P(a_0,\ldots,a_m)\sum_{a_{m+1},\ldots, a_{m+n-k}}Q(a_{m-k},\ldots,a_{m+n-k})}{R(a_{m-k},\ldots,a_m)}\\
&=\sum_{a_{0},\ldots, a_m}\frac{P(a_0,\ldots,a_m)R(a_{m-k},\ldots,a_m)}{R(a_{m-k},\ldots,a_m)}\\
&=\sum_{a_{0},\ldots, a_m}P(a_0,\ldots,a_m)=1.
\end{align*}
Hence~\eqref{probcompose} is an $(m+n-k)$-simplex in $\mathcal{P}_O$.

\begin{Proposition}\label{probabilities}
With these definitions, $\mathcal{P}_O$ is a compository.
\end{Proposition}

For similar reasons as for the proof of Proposition~\ref{metriccomp}, also the proof of this proposition will be deferred to Section~\ref{take2}.

\section{Gleaves}

As explained in the introduction, there are various natural examples of presheaves that are not sheaves, but still seem to possess more interesting structure than merely being presheaves: there exists a ``gluing operation'' which describes a canonical way of joining pairs of compatible local sections. We axiomatise the resulting notion of \emph{gleaf} in this section and show how compositories can be seen as particular kinds of gleaves. Both the metric space example and the joint probability distributions example given in the previous section have several variants all of which can be described as gleaves. 

\subsection{Gleaves on a lattice}\label{gleaveslattice}

When introducing sheaves, one usually starts by defining sheaves on topological spaces before moving on to the general definition of sheaves on sites. In a similar manner, we start with the definition of gleaves on the lattice of opens of a topological space, or, more generally, on any distributive lattice. Since any distributive lattice can be represented as a lattice of sets~\cite{stone}, we may consider, without loss of generality, a lattice of subsets $L\subseteq 2^X$ of some set $X$. With this in mind, we denote the lattice ordering by $\subseteq$ and the lattice operations by $\cap$ and $\cup$.

\begin{Definition}\label{littlegleaf}
Given a distributive lattice $L$, a \emph{gleaf} on $L$ is a presheaf $\Gamma:L^\op\to\Sets$ together with a \emph{gluing operation}
\be
\label{glop}
\xymatrix{ g_{U,V} : \Gamma(U)\times_{\Gamma(U\cap V)}\Gamma(V) \ar[r] & \Gamma(U\cup V) }
\ee
for every unordered pair $U,V\in L$, such that the following conditions hold:
\begin{enumerate}
\item \label{littlega} if $U\subseteq V$, then $g_{U,V}(\alpha,\alpha|_{V})=\alpha$ for all $\alpha\in\Gamma(U)$.
\item \label{littlegb} For $U',U,V\in L$ with $U'\subseteq U$ and $U'\cup V=U\cup V$,
\be
\vxymatrix{ \Gamma(U')\times_{\Gamma(U'\cap V)}\Gamma(V) \ar[rr]^(.6){g_{U',V}} \ar[dd]_{g_{U',V}} && \Gamma(U\cup V) \\\\
\Gamma(U\cup V) \ar[rr] && \Gamma(U)\times_{\Gamma(U\cap V)}\Gamma(V) \ar[uu]_{g_{U,V}} 
}\ee
\item \label{littlegc} For $U',U,V\in L$ with $U'\subseteq U$ and $U'\cap V=U\cap V$,
\be
\vxymatrix{ \Gamma(U) \times_{\Gamma(U\cap V)}\Gamma(V) \ar[dd]\ar[rr]^(.6){g_{U,V}} && \Gamma(U\cup V) \ar[dd]\\\\
\Gamma(U')\times_{\Gamma(U\cap V)} \Gamma(V) \ar[rr]_(.6){g_{U',V}} && \Gamma(U'\cup V)
}\ee
commutes.
\end{enumerate}
\end{Definition}

The axioms here are parallel to those of Definition~\ref{def:comp}. They axioms have a multitude of consequences which, however, will not be analysed in this section but rather when considering the definition of gleaves on a category. This includes, for example, associativity of the gluing operation.

\begin{Remark}
	These axioms are essentially equivalent to those of Dawid and Studen\'y for \emph{conditional products}~\cite{DS}. To wit, their \emph{projection} corresponds to restriction in the presheaf; their $\otimes$ is our gluing, guaranteed by \textbf{T1} to of type~\eqref{glop}; their \textbf{T2} is automatic in our setup, due to the use of \emph{unordered} pairs $U,V$; their \textbf{T3} is our~\ref{littlega}; their \textbf{T4} is our~\ref{littlegc}; their \textbf{T5} is our~\ref{littlegb}. The definition of~\cite{DS} is slightly more general in that the gluing operation may be defined only on a proper subset of all compatible pairs of local sections; this is what their \textbf{T6} is concerned with.
\end{Remark}

In the case of a sheaf, the canonical choice for $L$ is the lattice of opens $\mathcal{O}(X)$ of a topological space $X$. However, we do not see why this should necessarily likewise apply to gleaves, and in fact believe that other choices are sometimes more natural:

\begin{Example}
Let $L$ be the lattice of \emph{compact} subspaces of a Hausdorff space $X$, and for $U\in L$ let $\Gamma(U)$ be the set of all \emph{non-degenerate} metrics $U\times U\rightarrow \R_{\geq 0}\cup \{\infty\}$ which induce the given subspace topology. $\Gamma$ is a presheaf with the obvious restriction maps. For $d_U\in\Gamma(U)$ and $d_V\in\Gamma(V)$, we construct $d_{U\cup V} = g_{U,V}(d_U,d_V)$ by
\be\label{metric}
d_{U\cup V}(x,z) \defin \begin{cases}
d_U(x,z) & \text{if }x\in U,\:z\in U,\\
\inf_{y\in U\cap V}\mleft[d_U(x,y)+d_V(y,z)\mright] & \text{if }x\in U,\:z\in V,\\
\inf_{y\in U\cap V}\mleft[d_V(x,y)+d_U(y,z)\mright] & \text{if }x\in V,\:z\in U,\\
d_V(x,z) & \text{if }x\in V,\:z\in V. \end{cases}
\ee
Here, compactness guarantees that these infima are attained at some $y$ if $U\cap V\neq\emptyset$. Hence $d_{U\cup V}$ is a metric as well, i.e.~it is non-degenerate. It can also be shown that $d_{U\cup V}$ induces the given subspace topology on $U\cup V$, which implies that $d_{U\cup V}\in\Gamma(U\cup V)$.

Without any compactness requirement, the infimum in~\eqref{metric} is not necessarily attained at any point in the intersection. Take for example the set 
$$
X=\{u, v\}\cup \Nl
$$
with two distinct subsets
$$
U=\{u\}\cup \Nl,\qquad V=\{v\}\cup \Nl
$$
equipped with metrics $d_U$ and $d_V$ satisfying
\[
d_U(u,n)=\frac{1}{n},\qquad d_V(v,n)=\frac{1}{n}\quad\forall \;n\in\Nl,
\]
and otherwise arbitrary distances. Applying definition \eqref{metric} to $u$ and $v$ results in
\[
d_{U\cup V}(u,v)=\inf_{n\in\N}\mleft(d_U(u,n)+d_V(n,v)\mright)=\inf_{n\in\N}\mleft(\tfrac{1}{n}+\tfrac{1}{n}\mright)=0.
\]
Thus, the two distinct points $u$ and $v$ have zero distance.
\end{Example}

\subsection{Systems of bicoverings}
\label{bicoverings}

We would like to give a more general definition of gleaves which comprises both gleaves on lattices and compositories as special cases. 
To this end, we introduce the notion of a \emph{system of bicoverings} which are to gleaves what Grothendieck topologies are to sheaves.

\begin{Definition}
A \emph{system of bicoverings} on a category $\sC$ is a collection of cospans in $\sC$, called \emph{bicoverings}, which we draw with one vertical and one horizontal leg 
\[\xymatrix{
&b\ar[d]\\
a\ar[r]&c
}
\]
satisfying the following axioms:
\begin{enumerate}
\item \label{mono} The two legs of a bicovering are monomorphisms. 
\item Bicoverings can be completed to pullback squares.
\item \label{max}\emph{Maximal bicoverings:} for every $a\in\sC$, 
\[\xymatrix{
&a\ar@{=}[d]\\
a\ar@{=}[r]&a}\] is a bicovering.
\item\label{stabcomp} \emph{Stability under composition:}
given a diagram 
\[\xymatrix{
&a\times_c b\ar[r]\ar[d]&b\ar[d]\\
a'\ar[r]&a\ar[r]&c}
\]
in which both cospans are bicoverings, then so is the composed cospan.
\item \label{stabpb}
\emph{Stability under pullbacks:}
given a diagram
\[\xymatrix{
&&f^*(b)\ar[dr]\ar[dd]&\\
&&&b\ar[dd]\\
f^*(a)\ar[dr]\ar[rr]&&c'\ar[dr]|f&\\
&a\ar[rr]&&c}
\]
where both squares are pullbacks and the lower right cospan is a bicovering, then so is the upper left cospan. 
\end{enumerate}
\end{Definition}

Here, we do not distinguish the two legs of a cospan, i.e.~swapping the vertical and horizontal legs results in the same cospan. We do not assume that $\sC$ has all pullbacks; correspondingly, axiom~\ref{stabpb} only applies when the pullbacks exist.

\begin{Example}\label{sitebicovering}
Let $(\sC,J)$ be a site, where $\sC$ is a category with pullbacks. Then we define a cospan with monomorphic legs to be a bicovering if the sieve generated by it is a covering sieve. Axiom~\ref{mono} is satisfied by assumption. Axiom~\ref{max} holds since principal sieves are covering sieves. Axiom~\ref{stabcomp} follows from the transitivity axiom of Grothendieck topologies. Axiom~\ref{stabpb} follows from the stability axiom of Grothendieck topologies.
\end{Example}

\begin{Example}\label{examplepullback}
In the simplex category $\Delta$, we take a bicovering to be a cospan of the form 
\be\label{bicomp}\vxymatrix{
&[n]\ar[d]^{t_n}\\
[m]\ar[r]_{s_m}&[j]}
\ee
where $n+m\geq j$, corresponding to joint surjectivity of the two arrows. 
\end{Example}

\begin{Example}
\label{exadh}
Let $\sC$ be an adhesive category~\cite{adhesive}. Define a cospan with monomorphic legs to be a bicovering if its pullback square is also a pushout. Then Axioms~\ref{max} and~\ref{stabcomp} are by general properties of pullbacks and pushouts, while Axiom~\ref{stabpb} holds since pullbacks of monomorphisms are monomorphisms, and pushouts along monomorphisms in an adhesive category are van Kampen squares.
\end{Example}

Our first observation is a partial converse to Axiom~\ref{stabcomp}:

\begin{Lemma}
In a diagram 
\[\xymatrix{
&a\times_c b\ar[r]\ar[d]&b\ar[d]\\
a'\ar[r]&a\ar[r]^f&c}
\]
where the right and the composed cospan are bicoverings, then so is the left cospan.
\end{Lemma}

\begin{proof}
The left square of the diagram
\[\xymatrix@-=16pt{
&&a\times_c b\ar[dd]\ar[dr]&\\
&&&b\ar[dd]\\
a'\ar[rr]\ar@{=}[dr]&&a\ar[dr]|f&\\
&a'\ar[rr]&&c
}\]
is a pullback since $f$ is a monomorphism. The claim now follows from stability under pullback.
\end{proof}

\begin{Definition}
\label{morphbicov}
A \emph{morphism of bicoverings} is a diagram of the form
\[\xymatrix{
&&b'\ar[dr]^{q_b}\ar[dd]&\\
&&&b\ar[dd]\\
a'\ar[rr]\ar[dr]_{q_a}&&c'\ar[dr]|*+<4pt>{\scriptstyle{{q_c}}}&\\
&a\ar[rr]&&c
}\]
which can be completed to a diagram of the form 
\be\label{composition}\vxymatrix{
&&b'\ar[dr]^{q_b}\ar'[d][dd]&\\
&a\times_cb\ar[dd]\ar@/^4pt/[ur]\ar@/_4pt/[dl]\ar[rr]&&b\ar[dd]\\
a'\ar'[r][rr]\ar[dr]_{q_a}&&c'\ar[dr]|*+<4pt>{\scriptstyle{{q_c}}}&\\
&a\ar[rr]&&c
}\ee
\end{Definition}

The map $q_c$ completely determines the maps $q_a$ and $q_b$, since the legs of a bicovering are monomorphisms.
The existence of the two additional arrows in~\eqref{composition} is equivalent to the requirement that the induced map $a'\times_{c'} b'\rightarrow a\times_c b$ has a right inverse.

\subsection{Gleaves on a category with bicoverings}
\label{categorybicoverings}
Let $\sC$ be a category equipped with a system of bicoverings and $\mathsf{D}$ a category with pullbacks. For any $\mathsf{D}$-valued presheaf $\Gamma:\sC^{\op}\to\mathsf{D}$ and any pullback square
\[
\xymatrix{ a\times_c b\ar[r]\ar[d] & b\ar[d] \\
a\ar[r] & c}
\]
the restriction maps determine a natural arrow
\[
\Gamma(c) \longrightarrow \Gamma(a)\times_{\Gamma(a\times_c b)}\Gamma(b) .
\]
The basic idea behind gleaves is the existence of a ``gluing operation'' which constructs a ``glued'' local section in $\Gamma(c)$ from a compatible pair of local sections in $\Gamma(a)\times_{\Gamma(a\times_c b)}\Gamma(b)$.

Unlike in Section~\ref{gleaveslattice}, we now omit the subscript of a gluing operation indexing its components.

\begin{Definition}
\label{defbiggleaf}
A \emph{gleaf} on $\sC$ with values in $\mathsf{D}$ is a pair $(\Gamma,g)$ consisting of a functor $\Gamma:\sC^{\op}\to\mathsf{D}$ together with a \emph{gluing operation}
\[\xymatrix{ 
g \: : \: \Gamma(a)\times_{\Gamma(a\times_c b)} \Gamma(b) \ar[r] & \Gamma(c)
}\]
for every bicovering 
\be\label{bicovering}
\vxymatrix{
&b\ar[d]\\
a\ar[r]&c}
\ee
satisfying the following conditions:
\begin{enumerate}
\item\label{identityax} \emph{Identity axiom:}
if the bicovering is of the form
\[\vxymatrix{
&b\ar[d]\\
a\ar@{=}[r]&a}\qquad\text{ or }\qquad\vxymatrix{
&b\ar@{=}[d]\\
a\ar[r]&b}
\]
then
$
g=\pi_1:\Gamma(a)\times_{\Gamma(a)}\Gamma(b)\longrightarrow \Gamma(a)
$ or
$
g=\pi_2:\Gamma(a)\times_{\Gamma(b)}\Gamma(b)\longrightarrow \Gamma(b)
$,
respectively.

\item\label{BFax} \emph{Back-and-forth axiom:} 
given a diagram
\[\xymatrix{
&a\times_c b\ar[r]\ar[d]&b\ar[d]\\
a'\ar[r]&a\ar[r]&c}
\]
where both cospans are bicoverings, then the diagram 
\[\xymatrix{
\Gamma(a')\times_{\Gamma(a'\times_c b)}\Gamma(b)\ar[r]^(.6)g\ar[d]_g&\Gamma(c)\\
\Gamma(c)\ar[r]&\Gamma(a)\times_{\Gamma(a\times_c b)}\Gamma(b)\ar[u]_g
}\]
commutes. 
\item \label{natural}\emph{Partial naturality axiom:}
For any morphism of bicoverings 
\be\label{naturality}\vxymatrix{
&&b'\ar[dr]\ar'[d][dd]&\\
&a\times_cb\ar[dd]\ar@/^4pt/[ur]\ar@/_4pt/[dl]\ar[rr]&&b\ar[dd]\\
a'\ar'[r][rr]\ar[dr]&&c'\ar[dr]&\\
&a\ar[rr]&&c
}\ee
the induced diagram
\[\xymatrix{
\Gamma(a)\times_{\Gamma(a\times_c b)}\Gamma(b)\ar[rr]^-{g}\ar[dd]&&\Gamma(c)\ar[dd]\\\\
\Gamma(a')\times_{\Gamma(a'\times_{c'}b')}\Gamma(b')\ar[rr]^-{g}&&\Gamma(c')
}\]
commutes.
\end{enumerate}
\end{Definition}

The following development is completely parallel to that of Section~\ref{compositoryintro} on compositories. Our first observation is that restricting the gluing of two local sections recovers these original sections:

\begin{Lemma}
\label{recover}
For any bicovering \eqref{bicovering}, the induced diagram 
\be\label{identity}
\vxymatrix{
&\Gamma(a)&\\
\Gamma(a)\times_{\Gamma(a\times_c b)}\Gamma(b)\ar[rr]|-{g}\ar[ur]^-{\pi_1}\ar[dr]_-{\pi_2}&&\Gamma(c)\ar[dl]\ar[ul]\\
&\Gamma(b)&
}\ee
commutes.
\end{Lemma}

\begin{proof}
Since the map $a\rightarrow c$ is a monomorphism, both squares in
\[\xymatrix@-=16pt{
&&a\times_c b\ar[dd]\ar[dr]\\
&&&b\ar[dd]\\
a\ar@{=}[dr]\ar@{=}[rr]&&a\ar[dr]&\\
&a\ar[rr]&&c
}\]
are pullback squares, and hence the back cospan is a bicovering as well. Since the induced map between the resulting pullbacks is $\id_{a\times_c b}$, we are dealing with a morphism of bicoverings. Hence the partial naturality axiom applies and we obtain
\[\xymatrix@-=16pt{
\Gamma(a)\times_{\Gamma(a\times_c b)}\Gamma(b)\ar[rr]^-{g}\ar[dd]&&\Gamma(c)\ar[dd]\\\\
\Gamma(a)\times_{\Gamma(a\times_{c}b)}\Gamma(a\times_c b)\ar[rr]^-{g}&&\Gamma(a)
}\]
By the identity axiom, the lower horizontal arrow $g$ coincides with $\pi_1$. Therefore also the lower composition coincides with the projection $\pi_1$. Comparing this with the upper composition results in the upper triangle of diagram \eqref{identity}. Commutativity of the lower triangle is proven in an analogous way.
\end{proof}

\begin{Lemma}[two-step rule]\label{2step}
Given a diagram
\[\xymatrix{
&&b'\ar[d]\\
&a\times_c b\ar[r]\ar[d]&b\ar[d]\\
a'\ar[r]&a\ar[r]&c
}\]
where all three cospans are bicoverings, then the induced diagram
\be\label{2stepgleaf}\vcenter{\vbox{\xymatrix@R=.5cm{
\Gamma(a')\times_{\Gamma(a'\times_c b)}\Gamma(b)\ar[rr]\ar[dd]_g&&\Gamma(a')\times_{\Gamma(a'\times_c b)}\Gamma(a\times_c b)\times_{\Gamma(a'\times_c b)}\Gamma(b)\ar[dd]^{g\times_{\id}\id}&&\\\\
\Gamma(c)&&\Gamma(a)\times_{\Gamma(a\times_c b)}\Gamma(b)\ar[ll]|*+<4pt>{\scriptstyle{{g}}}\\\\
\Gamma(a)\times_{\Gamma(a\times_c b')}\Gamma(b')\ar[uu]^g\ar[rr]&&\Gamma(a)\times_{\Gamma(a\times_c b')}\Gamma(a\times_c b)\times_{\Gamma(a\times_c b')}\Gamma(b')\ar[uu]_{\id\times_{\id}g}
}}}\ee
commutes. 
\end{Lemma}

Intuitively, the upper half of this diagram states that a local section over $a'$ can be glued with a local section over $b$ by first gluing the former with the restriction of the latter to $a\times_c b$ and then gluing the result with the original section over $b$. A similar interpretation applies to the lower half of the diagram. 

\begin{proof}
Since the two parts of the diagram~\eqref{2stepgleaf} are equivalent under exchanging the two legs of all bicoverings, it is sufficient to prove commutativity of the upper half only. In the diagram
\[\xymatrix@=16pt{
&&a\times_c b\ar[dd]\ar[dr]\\
&&&b\ar[dd]\\
a'\ar@{=}[dr]\ar[rr]&&a\ar[dr]&\\
&a'\ar[rr]&&c
}\]
both squares are pullbacks since the map $a\rightarrow c$ is a monomorphism. 
Since the induced map between the resulting pullbacks is $\id_{a'\times_cb}$, we are dealing with a morphism of bicoverings.
Applying partial naturality gives 
\[\xymatrix{
\Gamma(a')\times_{\Gamma(a'\times_c b)}\Gamma(b)\ar[rr]^-g\ar[dd]&&\Gamma(c)\ar[dd]\\\\
\Gamma(a')\times_{\Gamma(a'\times_{c}b)}\Gamma(a\times_c b)\ar[rr]^-g&&\Gamma(a)
}\]
In the diagram
\[\xymatrix@!@R=-5cm@C=-3cm{
\Gamma(a')\times_{\Gamma(a'\times_c b)}\Gamma(b)\ar[rr]^-g\ar[dd]\ar[dr]|*+<4pt>{\scriptstyle{{g}}}&&\Gamma(c)\\
&\Gamma(c)\ar[dr]&\\
\Gamma(a')\times_{\Gamma(a'\times_{c}b)}\Gamma(a\times_c b)\times_{\Gamma(a\times_c b)}\Gamma(b)\ar[rr]^-{g\times_\id\id}&&\Gamma(a)\times_{\Gamma(a\times_c b)}\Gamma(b)\ar[uu]_-g
}\]
the upper triangle commutes because of the back-and-forth axiom, while the lower triangle commutes as a consequence of the partial naturality diagram above and~\eqref{identity}. 
\end{proof}

\begin{Proposition}[Associativity]
\label{assocprop}
Given a diagram
\[\xymatrix{
&&b'\ar[d]\\
&a\times_c b\ar[d]\ar[r]&b\ar[d]\\
a'\ar[r]&a\ar[r]&c
}\]
where all three cospans are bicoverings, then the induced diagram
\be\label{assoc}
\vxymatrix{
\Gamma(a')\times_{\Gamma(a'\times_c b)} \Gamma(a\times_c b)\times_{\Gamma(a\times_c b')} \Gamma(b') \ar[dd]_-{\id\times_{\id}g} \ar[rr]^-{g\times_{\id}\id}&& \Gamma(a)\times_{\Gamma(a\times_c b')}\Gamma(b')\ar[dd]^-g\\\\
\Gamma(a')\times_{\Gamma(a'\times_c b)} \Gamma(b)\ar[rr]^-g && \Gamma(c) }
\ee
commutes.
\end{Proposition}

\newcommand{\circnum}[1]{\raisebox{.5pt}{\textcircled{\raisebox{-.9pt}{#1}}}}

\begin{proof}
We prove this by showing commutativity of the diagram
\[\xymatrix{
&\Gamma(a')\times_{\Gamma(a'\times_c b)} \Gamma(a\times_c b)\times_{\Gamma(a\times_c b')} \Gamma(b')\ar[dl]_{\cong}\ar[r]^-{g\times_{\id}\id}\ar@{}[d]|{\circnum{3}}&\Gamma(a)\times_{\Gamma(a\times_c b')}\Gamma(b')\ar[dd]^g\ar[dl]\\
*\txt{$\Gamma(a')\times_{\Gamma(a'\times_c b)} \Gamma(a\times_c b)$\\
$\times_{\Gamma(a\times_c b)}\Gamma(a\times_c b)\times_{\Gamma(a\times_c b')} \Gamma(b')$}\ar[dr]_{g\times_{\id} g}\ar[r]^-{g\times_{\id}\id}&\Gamma(a)\times_{\Gamma(a\times_c b)}\Gamma(a\times_c b)\times_{\Gamma(a\times_c b')}\Gamma(b')\ar[d]^-{\id\times_{\id}g}\ar@<2.5ex>@{}[dl]|(.18){\circnum{1}}\ar@<4.5ex>@{}[dr]_(.7){\circnum{2}}\\
&\Gamma(a)\times_{\Gamma(a\times_c b)}\Gamma(b)\ar[r]_-{g}&\Gamma(c)
}
\]
in which the upper right composition is the one of~\eqref{assoc}, while the lower left one can be thought of as a ``diagonal'' in~\eqref{assoc}: we are about to show that it coincides with this diagram's upper right composition, and thanks to its invariance under swapping $a$ and $b$, it is then also equal to the diagram's lower left composition.

Subdiagram $\circnum{1}$ commutes trivially, while \circnum{2} is an instance of the two-step rule Lemma~\ref{2step}. Since all arrows of subdiagram \circnum{3} act trivially on the last component $\Gamma(b')$, this part reduces to 
\[\xymatrix{
\Gamma(a')\times_{\Gamma(a'\times_c b)} \Gamma(a\times_c b)\ar[r]^-g\ar[d]_{\cong}&\Gamma(a)\ar[d]\\
\Gamma(a')\times_{\Gamma(a'\times_c b)} \Gamma(a\times_c b)\times_{\Gamma(a\times_c b)}\Gamma(a\times_c b)\ar[r]^-{g\times_{\id}\id}&\Gamma(a)\times_{\Gamma(a\times_c b)}\Gamma(a\times_c b)
}
\]
This diagram can be shown to commute by postcomposing with the projections $\pi_1$ and $\pi_2$. This is trivial for $\pi_1$, for which both ways of composing the arrows yield $g$. In the case of $\pi_2$, commutativity holds since according to Lemma~\ref{recover}, saying that restricting the glued section in $\Gamma(a)$ back to $\Gamma(a\times_c b)$ recovers the original section given there.
\end{proof}

This ends our present development of the general theory of gleaves on categories with bicoverings. We now move on to discuss some general classes of examples.

\begin{Example}
Let $L$ be a distributive lattice. We regard $L$ as a category with bicoverings given by those cospans 
\[\xymatrix{
&V\ar[d]^{\rotatebox{-90}{$\scriptstyle{\subseteq}$}}\\
U\ar[r]_{\rotatebox{0}{$\scriptstyle{\subseteq}$}}&W
}\]
which satisfy $U\cup V=W$. It is easy to show that the axioms for a system of bicoverings hold; the stability under pullbacks is exactly distributivity. On this base category, the definition of gleaf~\ref{defbiggleaf} reduces to the one of a gleaf on a distributive lattice~\ref{littlegleaf}.
\end{Example}

\begin{Example}[Base change for gleaves]\label{basegleaves}
Let $\sC$ and $\sC'$ be categories equipped with systems of bicoverings and $\mathsf{D}$ a category with pullbacks.
Given a gleaf $(\Gamma, g)$, where $\Gamma:\sC^{\op}\rightarrow \mathsf{D}$, and a functor $\Xi:\sC'\rightarrow\sC$ which preserves bicoverings and pullbacks, then the composite $\Gamma\,\Xi$ carries an induced structure of gleaf with a gluing operation whose components are components of $g$.
\end{Example}

\begin{Theorem}
Given a site $(\sC, J)$, a $\mathsf{D}$-valued sheaf $\Gamma:\sC^{\op}\rightarrow\mathsf{D}$ on $(\sC, J)$ is a gleaf in a unique way with respect to the bicoverings given by cospans
\[\xymatrix{
&b\ar@{^{(}->}[d]\\
a\ar@{^{(}->}[r]&c
}\]
which generate a covering sieve (Example~\ref{sitebicovering}).
\end{Theorem}

\begin{proof}
From the sheaf condition, we know that 
$$
\Gamma(c)\longrightarrow \Gamma(a)\times_{\Gamma(a\times_c b)}\Gamma(b)
$$
is an isomorphism. Since the gluing operation is required to be a right inverse of this map, it is automatically unique and given by the inverse isomorphism.
The axioms of gleaves can be easily seen to hold by postcomposing each required diagram with $\pi_1 g^{-1}$, respectively $\pi_2 g^{-1}$, and using the fact that $\pi_1 g^{-1}$ and $\pi_2 g^{-1}$ are jointly monic. 
\end{proof}

\begin{Theorem}\label{gleavescompo}
Gleaves $\Delta^{\op}\rightarrow \Sets$ on the simplex category $\Delta$ with bicoverings of the form
\be\label{bicomp2}\vxymatrix{
&[n]\ar[d]^{t_{n}}\\
[m]\ar[r]_{s_m}&[j]
}\qquad(n+m\geq j)\ee
are compositories and vice versa.
\end{Theorem}

The proof of this theorem requires the following lemma.

\begin{Lemma}\label{complab}
\begin{enumerate}
\item\label{compla} For $n+m\geq j$ and $n'+m'\geq j'$, the diagram 
\be\label{surjective}\vxymatrix{
&&[n']\ar[dr]^{q|_{[n']}}\ar[dd]_{t_{n'}}&\\
&&&[n]\ar[dd]^{t_n}\\
[m']\ar[rr]^{s_{m'}}\ar[dr]_{q|_{[m']}}&&[j']\ar[dr]|*+<4pt>{\scriptstyle{{q}}}&\\
&[m]\ar[rr]_{s_m}&&[j]
}\ee
is a morphism of bicoverings iff $q|_{[m']\cap[n']}:\mathrm{im}(s_{m'})\cap\mathrm{im}(t_{n'})\longrightarrow \mathrm{im}(s_m)\cap\mathrm{im}(t_n)$ is surjective.
\item\label{complb} Any morphism of bicoverings is a composition of those of the form 
\be\label{bidegeneracy}\vxymatrix{
&[n]\ar@{=}[dr]\ar[d]_{t_{n}}&\\
[m]\ar[r]^{s_{m}}\ar[dr]_{\eta_i}&[j]\ar[dr]|*+<4pt>{\scriptstyle{{\eta_i}}}&[n]\ar[d]^{t_n}\\
&[m-1]\ar[r]_{s_{m-1}}&[j-1]
}\quad\vxymatrix{
&[n]\ar[dr]^{\eta_{i+n-j}}\ar[d]_{t_{n}}&\\
[m]\ar[r]^{s_{m}}\ar@{=}[dr]&[j]\ar[dr]|*+<4pt>{\scriptstyle{{\eta_i}}}&[n-1]\ar[d]^{t_{n-1}}\\
&[m]\ar[r]_{s_m}&[j-1]
}\quad\vxymatrix{
&[n]\ar[dr]^{\eta_{i+n-j}}\ar[d]_{t_{n}}&\\
[m]\ar[r]^{s_{m}}\ar[dr]_{\eta_i}&[j]\ar[dr]|*+<4pt>{\scriptstyle{{\eta_i}}}&[n-1]\ar[d]^{t_{n-1}}\\
&[m-1]\ar[r]_{s_{m-1}}&[j-1]
}
\ee
which correspond to the degeneracy conditions~\eqref{compdeg1}--\eqref{compdeg3}, and those of the form 
\be\label{biface}\vxymatrix{
&[n]\ar@{=}[dr]\ar[d]_{t_{n}}&\\
[m-1]\ar[r]^{s_{m-1}}\ar[dr]_{\partial_i}&[j-1]\ar[dr]|*+<4pt>{\scriptstyle{{\partial_i}}}&[n]\ar[d]^{t_n}\\
&[m]\ar[r]_{s_m}&[j]
}\qquad\vxymatrix{
&[n-1]\ar[dr]^{\partial_{i+n-j}}\ar[d]_{t_{n-1}}&\\
[m]\ar[r]^-{s_{m}}\ar@{=}[dr]&[j-1]\ar[dr]|*+<4pt>{\scriptstyle{{\partial_i}}}&[n]\ar[d]^{t_n}\\
&[m]\ar[r]_{s_m}&[j]
}\ee
\end{enumerate}
which correspond to the face map conditions~\eqref{faceS} and~\eqref{faceT}.
\end{Lemma}

\begin{proof}
We start with~\ref{compla}. The two additional arrows $[m']\leftarrow[m+n-j]\rightarrow[n']$ required for being a morphism of bicoverings exist iff $q|_{[m']\cap[n']}$ has a right inverse. In $\Delta$, this is equivalent to surjectivity. 

For~\ref{complb}, we first claim that any morphism of bicoverings can be split into a composition of a morphism with surjective components followed by a morphism with injective components as in the diagram 
\be\label{surinj}
\vcenter{\vbox{\xymatrix@!{
&&[n']\ar@{->>}[dr]\ar[dd]_{t_{n'}}\\
&&&[\widehat{n}]\ar[dd]_{t_{\widehat{n}}}\ar@{^{(}->}[dr]\\
[m']\ar[rr]^{s_{m'}}\ar@{->>}[dr]&&[j']\ar@{->>}[dr]|*+<4pt>{\scriptstyle{{q_1}}}&&[n]\ar[dd]_{t_n}\\
&[\widehat{m}]\ar[rr]^{s_{\widehat{m}}}\ar@{^{(}->}[dr]&&[\widehat{j}]\ar@{^{(}->}[rd]|*+<4pt>{\scriptstyle{{q_2}}}\\
&&[m]\ar[rr]^{s_m}&&[j]
}}}\ee
To see this, we decompose $q$ into two parts using its image factorisation,
\[
\xymatrix{[j']=\mathrm{dom}(q)\ar@{->>}[r]^-{q_1}&\mathrm{im}(q)\ar@{^{(}->}[r]^-{q_2}&\mathrm{cod}(q)=[j]}
\]
We then define $[\widehat{m}]$ to be the pullback $q_2^{-1}([m])$, and similarly $[\widehat{n}]\defin q_2^{-1}([n])$. In particular, the pair $(s_{\widehat{m}}, t_{\widehat{n}})$ bicovers $[\widehat{j}]$. 
The maps $[m']\twoheadrightarrow [\widehat{m}]$ and $[n']\twoheadrightarrow[\widehat{n}]$ are then defined by the universal property of pullbacks. 
We show surjectivity of $[m']\twoheadrightarrow [\widehat{m}]$; a similar proof applies to $[n']\twoheadrightarrow[\widehat{n}]$. For any element $v\in [\widehat{m}]$, there exists a $v'\in [j']$ such that $q_1(v')=s_{\widehat{m}}(v)$. Then either $v'\in \mathrm{im}(s_{m'})$, in which case we are done, or $v'\in\mathrm{im}(t_{n'})$. In the latter case we obtain that $q_1(v')=s_{\widehat{m}}(v)\in \mathrm{im}(t_{\widehat{n}})$, and  thus $s_{\widehat{m}}(v)\in\mathrm{im}(s_{\widehat{m}})\cap \mathrm{im}(t_{\widehat{n}})$. By the assumed surjectivity of $q|_{[n']\cap[m']}$, there exists $v''\in\mathrm{im}(s_{m'})\cap\mathrm{im}(t_{n'})$ such that $q(v'')=q_2(s_{\widehat{m}}(v))$. Since $q_2$ is injective, $q_1(v'')=s_{\widehat{m}}(v)$. Because $v''\in\mathrm{im}(s_{m'})$, this gives a preimage in $[m']$ of the original $v\in[\widehat{m}]$, as desired.

We have shown in passing that $q_1|_{[m']\cap[n']}: \mathrm{im}(s_{m'})\cap \mathrm{im}(t_{n'})\longrightarrow  \mathrm{im}(s_{\widehat{m}})\cap \mathrm{im}(t_{\widehat{n}})$ is surjective. Moreover, $q_2|_{[\widehat{m}]\cap[\widehat{n}]}: \mathrm{im}(s_{\widehat{m}})\cap \mathrm{im}(t_{\widehat{n}})\longrightarrow  \mathrm{im}(s_{m})\cap \mathrm{im}(t_n)$ is bijective thanks to the given assumption on $q$. 

We now decompose the lower right part of~\eqref{surinj} into morphisms of the form~\eqref{biface}. This we do by induction on $j-\widehat{j}$. In the base case $j=\widehat{j}$, we necessarily have $q_2=\id_{[j]}$, so that there is nothing to be done. For the induction step, we pick any $i\in[j]\setminus \mathrm{im}(q_2)$, so that $q_2$ can be factored as $\partial_i\widehat{q}_2 $. Since $  \mathrm{im}(s_m)\cap  \mathrm{im}(t_n)\subseteq\mathrm{im}(q_2)$, we either have $i\not\in \mathrm{im}(t_{n})$, which means that $i<j-n$, or $i\not\in \mathrm{im}(s_{m})$, which means that $i>m$. In the second case we obtain 
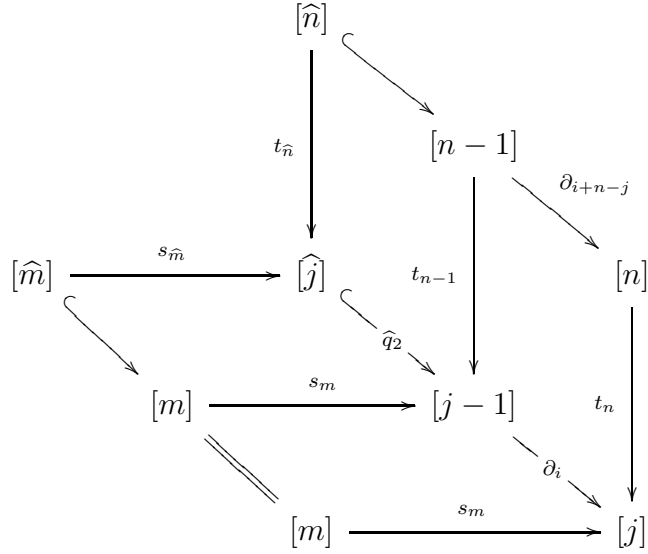
\[\xymatrix{
&&[\widehat{n}]\ar@{^{(}->}[dr]\ar[dd]_{t_{\widehat{n}}}\\
&&&[n-1]\ar[dd]_{t_{n-1}}\ar[dr]^{\partial_{i+n-j}}\\
[\widehat{m}]\ar[rr]^{s_{\widehat{m}}}\ar@{^{(}->}[dr]&&[\widehat{j}]\ar@{^{(}->}[dr]|*+<4pt>{\scriptstyle{\widehat{q}_2}}&&[n]\ar[dd]_{t_n}\\
&[m]\ar[rr]^{s_{m}}\ar@{=}[dr]&&[j-1]\ar[rd]|*+<4pt>{\scriptstyle{{\partial_i}}}\\
&&[m]\ar[rr]^{s_m}&&[j]
}\]
so that the claim follows from the induction assumption applied to $\widehat{q}_2$. The first case is analogous. 

We now decompose the first part of~\eqref{surinj} into morphisms of the form~\eqref{bidegeneracy} by induction on $j'-\widehat{j}$. The base case $j'=\widehat{j}$ is trivial. For the induction step, we pick an $i\in [j']$ such that $q_1(i)=q_1(i+1)$. We then have three different cases: 
\begin{enumerate}
\item $i<j'-n'$.
\item $j'-n'\leq i<m' $.
\item $m'\leq i$.
\end{enumerate}
We illustrate the proof for the first case, obtaining the diagram
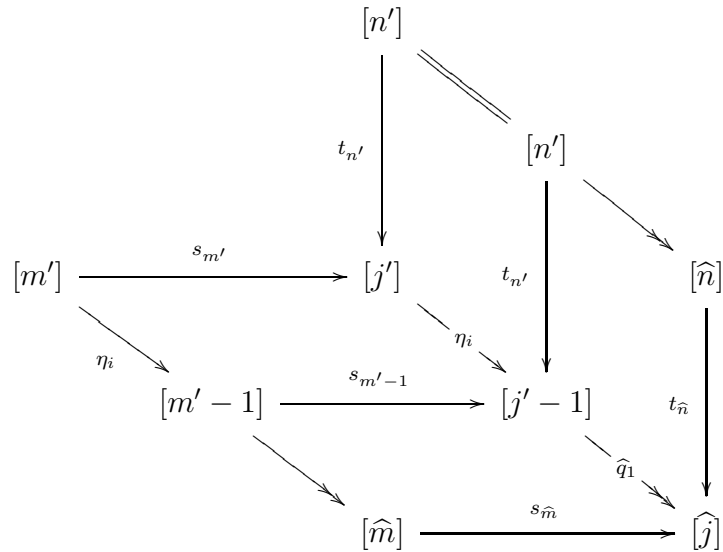
\[\xymatrix{
&&[n']\ar@{=}[dr]\ar[dd]_{t_{n'}}\\
&&&[n']\ar[dd]_{t_{n'}}\ar@{->>}[dr]\\
[m']\ar[rr]^{s_{m'}}\ar[dr]_{\eta_i}&&[j']\ar[dr]|*+<4pt>{\scriptstyle{{\eta_i}}}&&[\widehat{n}]\ar[dd]_{t_{\widehat{n}}}\\
&[m'-1]\ar[rr]^{s_{m'-1}}\ar@{->>}[dr]&&[j'-1]\ar@{->>}[rd]|*+<4pt>{\scriptstyle{{\widehat{q}_1}}}\\
&&[\widehat{m}]\ar[rr]^{s_{\widehat{m}}}&&[\widehat{j}]
}\]
The claim follows from the induction assumption applied to $\widehat{q}_1$. The two other cases can be treated in an analogous way.\qedhere
\end{proof}

\begin{proof}[Proof of Theorem~\ref{gleavescompo}]
We need to show that the compository axioms are equivalent to the gleaf axioms with respect to bicoverings of the form~\eqref{bicomp2}. The identity axiom for compositories corresponds exactly to the identity axiom for gleaves and similarly for the back-and-forth axiom.
Compatibility of composition with degeneracy maps is equivalent to naturality for morphisms of bicoverings of the form~\eqref{bidegeneracy}, while compatibility with face maps corresponds to~\eqref{biface}. 

On the other hand, Lemma \ref{complab} implies that these compatibility conditions are sufficient to guarantee naturality with respect to all morphisms of bicoverings.
\end{proof}

\begin{Remark}
One can also try to consider gleaves not just on $\Delta$, but also on other categories of geometric shapes, such as on the globe category or the cube category, and see whether they might have a meaningful interpretation as higher categories. This does not work for the globe category, which does not have any non-trivial pullbacks. So far, we have not investigated the case of the cube category any further.
\end{Remark}

\subsection{Categories of gleaves}
\label{categorygleaves}

We would like to turn the collection of $\mathsf{D}$-valued gleaves on a category $\sC$ with bicoverings into a category itself. 

\begin{Definition}
A morphism of gleaves $(\Gamma, g)\rightarrow(\Gamma', g')$ is a natural transformation $f:\Gamma\rightarrow \Gamma'$ such that 
\[\xymatrix@R=1.5cm@C=1.5cm{
\Gamma(a)\times_{\Gamma(a\times_c b)}\Gamma(b)\ar[r]^(.65)g\ar[d]_{f_a\times_{f_{a\times_cb}}f_b}&\Gamma(c)\ar[d]^{f_c}\\
\Gamma'(a)\times_{\Gamma'(a\times_c b)}\Gamma'(b)\ar[r]^(.7){g'}&\Gamma'(c)
}\]
commutes.
\end{Definition}
\noindent With this definition, the collection of gleaves $(\Gamma:\sC^{\op}\rightarrow\mathsf{D}, g)$ forms a category $\mathsf{Gl}(\sC,\mathsf{D})$. 

We now reproduce an argument communicated to us by Peter Johnstone which shows that a category of gleaves $\mathsf{Gl}(\sC,\Sets)$ is not necessarily a topos.

\begin{Theorem}[Johnstone]\label{johnstonetheorem}
There exists $\sC$ for which $\mathsf{Gl}(\sC,\Sets)$ is not cartesian closed, but does have a subobject classifier.
\end{Theorem}

\begin{proof}
Consider $\sC$ to be the category formed by the poset with Hasse diagram
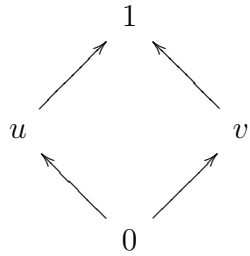
\[\xymatrix{
&1&\\
u\ar[ru]&&v\ar[lu]\\
&0\ar[ul]\ar[ur]&
}\]
which is the four-element Boolean algebra. As for any distributive lattice, we take as non-maximal bicoverings the three cospans 
\[\xymatrix{&v\ar[d]\\
u\ar[r]&1}\qquad
\xymatrix{
&0\ar[d]\\
u\ar@{=}[r]&u
}\qquad\xymatrix{
&v\ar@{=}[d]\\
0\ar[r]&v}\]
For a gleaf $(\Gamma, g)$ on $\sC$, the only non-trivial component of the gluing operation is $g:\Gamma(u)\times_{\Gamma(0)}\Gamma(v)\longrightarrow\Gamma(1)$. Hence such a gleaf is the same as a presheaf on $\sC$ together with a function $g:\Gamma(u)\times_{\Gamma(0)}\Gamma(v)\longrightarrow\Gamma(1)$ which is a right inverse of $\Gamma(1)\longrightarrow  \Gamma(u)\times_{\Gamma(0)}\Gamma(v)$. This function can be incorporated into the data of a presheaf upon adjoining an additional object $w$ to $\sC$ obtaining a new category $\sC'$  presented by
\be\label{sketch}\vxymatrix{
&1\ar@/^.8pc/[d]^r&\\
&w\ar@/^.8pc/[u]^i&&& ri=\id_w.\\
u\ar[ru]&&v\ar[lu]\\
&0\ar[ul]\ar[ur]\ar@{}[uu]|*{\circlearrowleft}&
}\ee
Then a gleaf on $\sC$ is the same as a presheaf on
$\sC'$ which satisfies the sheaf-like condition that the map $\Gamma(w)\longrightarrow  \Gamma(u)\times_{\Gamma(0)}\Gamma(v)$ is an isomorphism. The gluing operation is given by $g=\Gamma(r)$.
Note that the terminal object of $\sC'$ is $w$.

In conclusion, the category of gleaves $\mathsf{Gl}(\sC,\Sets)$ is equivalent to the category $\mathsf{Gl}$ of sheaves on $\sC'$, where $\{u\rightarrow w,v\rightarrow w\}$ is the unique non-trivial covering family. Since the pullback of this family along $r$ is not a covering family, $\sC'$ is not a site. 

A direct verification shows that all contravariant hom-functors are sheaves on $\sC'$. Our goal is to show that the hom-functor $h_1=\mathsf{Gl}(-,1)$ is not an exponentiable object in $\mathsf{Gl}$. Since the contravariant Yoneda embedding is cocontinuous, the diagram of hom-functors
\[\xymatrix{
h_0\ar[r]\ar[d] & h_v\ar[d]\\
h_u\ar[r] & h_w
}\]
is a pushout. Now if the hom-functor $h_1$ was an exponentiable object, then the functor $-\times h_1$ would be a left adjoint and hence preserve colimits. In particular, this would imply that also
\be\label{prodsquare}\vxymatrix{
h_0\times h_1\ar[r]\ar[d] & h_v\times h_1\ar[d]\\
h_u\times h_1\ar[r] & h_w\times h_1}
\ee
would have to be a pushout; however, a direct objectwise consideration shows that $h_0\times h_1\cong h_0$, $h_u\times h_1\cong h_u$ and $h_v\times h_1\cong h_v$, where each of these isomorphism is given by the corresponding product projection. On the other hand, $h_w\times h_1\cong h_1$ also by the product projection, because $h_w$ is terminal in $\mathsf{Gl}$. Since $h_1\not\cong h_w$, uniquness of pushouts up to isomorphism shows that~\eqref{prodsquare} is not a pushout. In conclusion, $\mathsf{Gl}$ does not have exponentials.

We now proceed to showing that $\mathsf{Gl}$ does have a subobject classifier which can be constructed just as in a Grothendieck topos, namely the presheaf $\Omega(x)\defin\{S\;|\;S\text{ is a closed sieve on } x\}$ for all $x\in \sC'$. A direct check shows that the sheaf condition $\Omega(w)\stackrel{\cong}{\longrightarrow} \Omega(u)\times_{\Omega(0)}\Omega(v)$ holds. The proof of \cite[Prop.~III.7.3]{MM} applies verbatim and shows that $\Omega$ is indeed a subobject classifier.
\end{proof}

Alternatively, one can prove that $\Omega$ is a subobject classifier using the reasoning of \cite[p.~551]{elephant}.

Some examples of non-cartesian-closed categories with a subobject classifier have been known previously~\cite{Engenes}. Johnstone's earlier example~\cite[Remark 8.3]{Johnstone1} is very similar to the example used in the above proof: his binary operation $b$ has the flavour of a gluing operation from which its arguments can be recovered by application of the unary operations $l$ and $r$, which are similar to restriction maps.

\section{Examples of gleaves}

We would now like to describe various concrete examples of gleaves, including a gleaf of metric spaces, similar in flavour to the compository of metrics (Section~\ref{metspacesI}), and a gleaf of joint probability distributions, corresponding to the compository of joint distributions (Section~\ref{secjpd}). We expect that many other geometrical structures of a ``global'' nature can be reformulated in this way.

\subsection{Metric spaces, take II}\label{takeII}
Since a metric lives on a set, we now take the base category to be $\sC\defin\Sets$. For any cospan with injective legs
\be\label{bicoverII}\vxymatrix{
&B\ar@{^{(}->}[d]\\
A\ar@{^{(}->}[r]&C
}\ee
we consider $A$ and $B$ to be subsets of $C$, thus omitting explicit mention of the inclusion maps. We take it to be a bicovering if the two legs are jointly surjective, i.e.~$A\cup B=C$, or equivalently use the definition from Example~\ref{exadh}.

The functor 
\begin{align*}
\begin{split}
\mathcal{M}:\Sets^{\op}&\rightarrow \Sets\\
A&\mapsto\{d:A\times A\rightarrow \Rl_{\geq 0}\cup\{\infty\}\:|\: d\text{ is a metric }\}\\
\mleft(\xymatrix{A\ar[r]^f & B}\mright)&\mapsto\mleft(\vcenter{\vbox{\xymatrix@R=.2cm@C=-1.7cm{ \mathcal{M}(B)\ar[rr] && \mathcal{M}(A) \\ d_B\ar@{|->}[rr] && d_A \\ & d_A(x,y)\defin d_B(f(x), f(y))}}}\mright)
\end{split}
\end{align*}
assigns to every set the collection of all ways of turning that set into a ``metric space''. Just as in Section~\ref{metspacesI}, what we mean by this is a set equipped with a distance function taking values in $\Rl_{\geq 0}\cup\{\infty\}$, satisfying the triangle inequality, assigning zero distance from any point to itself, and, depending on the reader's preference, the optional symmetry axiom.

The component of the gluing operation $g$ on a bicovering~\eqref{bicoverII} is defined as follows:
\[
\vcenter{\vbox{\xymatrix@R=.2cm{
g:\mathcal{M}(A)\times_{\mathcal{M}(A\times_C B)}\mathcal{M}(B) \ar[r] & \mathcal{M}(C) \\
(d_A, d_B) \ar@{|->}[r] & g(d_A,d_B)
}}}
\]
where $g(d_A,d_B)$ is the metric on $C$ given by
\be
\label{metrics}
g(d_A,d_B)(x,z) = \begin{cases}
d_A(x,z) & \text{if }x\in A,\:z\in A,\\
\inf_{y\in A\cap B}\mleft[d_A(x, y)+d_B(y,z)\mright] & \text{if }x\in A,\:z\in B,\\
\inf_{y\in A\cap B}\mleft[d_B(x, y)+d_A(y,z)\mright] & \text{if }x\in B,\:z\in A,\\
d_B(x,z) & \text{if }x\in B,\:z\in B.\\
\end{cases}
\ee
The infima are understood to be $\infty$ in case that $A\cap B=\emptyset$. Since these four cases overlap, it needs to be checked that the compatibility assumption
\[
d_A(w,w') = d_B(w,w') \qquad\forall w,w'\in A\cap B
\]
guarantees that this is well-defined. For example, if we take $x\in A$ and $z\in A\cap B$, then the result of applying the first case should coincide with the application of the second:
$$\inf_{y\in A\cap B}\mleft[d_A(x, y)+d_B(y,z)\mright]=\inf_{y\in A\cap B}\mleft[d_A(x, y)+d_A(y,z)\mright]=d_A(x,z).$$
Similar reasoning applies to all other overlap cases. It is clear that $g(d_A,d_B)|_A = d_A$ and $g(d_A,d_B)|_B = d_B$.

We now verify the triangle inequalities
\[
g(d_A,d_B)(x,z) \leq g(d_A,d_B)(x,y)+g(d_A,d_B)(y,z)
\]
in the case in which $x, z\in A$ and $y\in B$. All other cases are similar or simpler than this. We know that 
\begin{align*}
d_A(x,z)&\leq \inf_{w,w'\in A\cap B}\mleft[d_A(x, w)+d_A(w,w')+d_A(w',z)\mright]\\
&= \inf_{w,w'\in A\cap B}\mleft[d_A(x, w)+d_B(w,w')+d_A(w',z)\mright]\\
&\leq \inf_{w,w'\in A\cap B}\mleft[d_A(x, w)+d_B(w,y)+d_B(y,w')+d_A(w',z)\mright]\\
&=\inf_{w\in A\cap B} \mleft[d_A(x,w)+d_B(w,y)\mright] + \inf_{w'\in A\cap B} \mleft[d_B(y,w')+d_A(w',z)\mright],
\end{align*}
from which the assertion follows by the definition~\eqref{metrics}.

\begin{Proposition}
With these definitions, $(\mathcal{M},g)$ becomes a gleaf.
\end{Proposition}

Upon base change (Example~\ref{basegleaves}) along the inclusion functor $\Delta\to\Sets$, one recovers the compository of metric spaces from Section~\ref{metspacesI}.

\begin{proof}
The identity axiom holds trivially: if $A=C$ or $B=C$, then the first or last case of~\eqref{metrics} always applies.

In the back-and-forth axiom, we also have $A'\subseteq A$ with $A'\cup B=C$, start with metrics $d_{A'}$ and $d_B$, and need to show that
$$
g(d_{A'}, d_B)=g\mleft(g(d_{A'},d_B)|_{A}, d_B\mright).
$$
We exemplify the proof of this by evaluating on $x\in A'$ and $y\in B$. In this case, the right-hand side becomes
\begin{align*}
g\mleft(g(d_{A'},d_B)|_{A}, d_B\mright)(x,y)&=\inf_{w\in A\cap B}\mleft[g(d_{A'}, d_B)(x,w)+d_B(w,y)\mright]\\
&=\inf_{w\in A\cap B}\mleft[\inf_{w'\in A'\cap B}\mleft[d_{A'}(x,w')+d_B(w', w)\mright]+d_B(w,y)\mright]\\
&=\inf_{w'\in A'\cap B}\mleft[d_{A'}(x,w')+\inf_{w\in A\cap B}\mleft[d_B(w', w)+d_B(w,y)\mright]\mright]\\
&=\inf_{w'\in A'\cap B}\mleft[d_{A'}(x,w')+d_B(w',y)\mright] = g(d_{A'},d_B).
\end{align*}
The other half of the back-and-forth axiom with $B'\subseteq B$ is entirely analogous.

In the partial naturality axiom, we have another bicovering $C'=A'\cup B'$ and a map $q:C'\rightarrow C$ such that $q(A')\subseteq A$, $q(B')\subseteq B$ and $q|_{A'\cap B'}:A'\cap B'\rightarrow A\cap B$ is surjective. In our right-action notation, we then want to show that
$$
	g(d_A, d_B) q=g\mleft(d_A\, q|_{A'}, d_B\, q|_{B'}\mright).
$$
Again we sketch part of the proof of this by evaluating on $x\in A'$ and $y\in B'$. The right-hand side becomes
$$
	g\mleft(d_A\, q|_{A'}, d_B\, q|_{B'}\mright)(x,y)=\inf_{w'\in A'\cap B'}\mleft[d_A(q(x),q(w'))+d_B(q(w'), q(y))\mright],
$$
while the left-hand side looks like
$$
\mleft(g(d_A, d_B) q\mright)(x,y)=\inf_{w\in A\cap B}\mleft[d_A(q(x),w)+d_B(w, q(y))\mright].
$$
The claim follows from surjectivity of $q$ on the intersection.
\end{proof}

In combination with Example~\ref{basegleaves} and the fact that the inclusion functor $\Xi:\Delta\rightarrow \Sets$ preserves bicoverings and pullbacks, we have thus also proven Proposition~\ref{metriccomp}.

\subsection{Probability distributions, take II}\label{take2}

A similar development is possible for joint probability distributions: the example of Section~\ref{secjpd} can be turned into a gleaf on $\sC=\FinSets$ with values in $\mathsf{D}=\Sets$ in pretty much the same way as the metric space example. As in $\Sets$, we take the bicoverings on $\FinSets$ to be those cospans which have injective and jointly surjective legs.

As before, we fix a finite set of outcomes $O$, and now consider the functor
\begin{align*}
\mathcal{P}:\FinSets^{\op}&\rightarrow \Sets\\[5.pt]
A&\mapsto\{P_A\text{ probability measure on } O^A\}\\[5.pt]
\mleft(\xymatrix{A\ar[r]^f & B}\mright)&\mapsto\mleft(\vcenter{\vbox{\xymatrix@R=.2cm@C=-1.7cm{ \mathcal{P}(B)\ar[rr] && \mathcal{P}(A) \\ P_B\ar@{|->}[rr] && P_A \\ & P_A(\alpha)\defin \sum\limits_{\beta\in O^B \text{ s.t. }\beta f =\alpha}P_B(\beta)}}}\mright)
\end{align*}

The gluing operation on a bicovering~\eqref{bicoverII} is given by
\begin{align*}
g:\mathcal{P}(A)\times_{ \mathcal{P}(A\times_C B)} \mathcal{P}(B)&\rightarrow  \mathcal{P}(C) \\[5.pt]
(P_A, P_B)&\mapsto g(P_A, P_B)\\[5.pt]
g(P_A, P_B)(\gamma)&=\frac{P_A(\gamma|_{A})P_B(\gamma|_{B})}{P_{A\cap B}(\gamma|_{A\cap B})},
\end{align*}
where $P_{A\cap B}$ stands for $P_A|_{A\cap B}=P_B|_{A\cap B}$.
Concerning the case of vanishing denominator and the proof showing that this is a normalised distribution, statements analogous to those made in the compository version after~\eqref{probcompose} can be made.

The resulting $g(P_A,P_B)$ is the unique probability distribution which makes the variables in $A$ conditionally independent of those in $B$ given the values of those in $A\cap B$. In particular, we have $g(P_A,P_B)|_A = P_A$ and $g(P_A,P_B)|_B = P_B$ by construction.

\begin{Proposition}
With these definitions, $( \mathcal{P},g)$ is a gleaf.
\end{Proposition}

A variant of this result---with $\FinSets$ replaced by the lattice of finite subsets of a fixed set---is due to Dawid and Studen\'y~\cite[Proposition~3.1]{DS}. Upon base change (Example~\ref{basegleaves}) along the inclusion functor $\Delta\to\FinSets$, one recovers the compository of joint probability distributions from Section~\ref{secjpd}.

\begin{proof}
The identity axiom means that when $B\subseteq A$, then $g(P_A, P_B)=P_A$ for any compatible pair $(P_A, P_B)$. This is indeed the case since
$$
g(P_A, P_B)(\alpha)=\frac{P_A(\alpha)P_B(\alpha|_{B})}{P_{A\cap B}(\alpha|_{A\cap B})}=\frac{P_A(\alpha)P_B(\alpha|_{B})}{P_{B}(\alpha|_{B})}=P_A(\alpha).
$$

For the back-and-forth axiom, we need to show that whenever $A'\subseteq A$ such that $A'\cup B=A\cup B=C$, then 
$$
g(P_{A'}, P_B)=g\mleft(g(P_{A'},P_B)|_{A},P_B\mright).
$$
Evaluating the right-hand side on some $\gamma\in O^C$ results in
\begin{align*}
g\big(g(P_{A'},P_B)|_{A},P_B\big)(\gamma)&=g(P_{A'},P_B)|_{A}(\gamma|_{A})\cdot\frac{P_{B}(\gamma|_{B})}{P_{A\cap B}(\gamma|_{A\cap B})}
\\[5pt]
&=\sum\limits_{\gamma'\in O^C\text{ s.t. } \gamma'|_{A}=\gamma|_{A}}\frac{P_{A'}(\gamma'|_{A'})P_{B}(\gamma'|_{B})}{P_{A'\cap B}(\gamma'|_{A'\cap B}) }\cdot\frac{P_B(\gamma|_{B})}{P_{A\cap B}(\gamma|_{A\cap B})}\\[5pt]
&=\frac{P_{A'}(\gamma|_{A'})P_B(\gamma|_{B})} {P_{A'\cap B}(\gamma|_{A'\cap B}) P_{A\cap B}(\gamma|_{A\cap B})}\cdot \sum\limits_{\gamma'\in O^C\text{ s.t. } \gamma'|_{A}=\gamma|_{A}} P_{B}(\gamma'|_{B})
\\[5pt]
&=\frac{P_{A'}(\gamma|_{A'})P_B(\gamma|_{B})} {P_{A'\cap B}(\gamma|_{A'\cap B})}=g(P_{A'}, P_B)(\gamma)
\end{align*}
The fourth equality uses 
\[
P_{A\cap B}(\gamma|_{A\cap B})=\sum\limits_{\beta\in O^B\text{ s.t }\beta|_{A\cap B}=\gamma|_{A\cap B}}P_B(\beta)=\sum\limits_{\gamma'\in O^C\text{ s.t }\gamma'|_{A}=\gamma|_{A}}P_B(\gamma'|_{B}),
\]
which follows from the bijective correspondence between $\beta$ and $\gamma'$ (sheaf condition).

In the naturality axiom, we have another bicovering $C'=A'\cup B'$ and a map $q:C'\rightarrow C$ such that $q(A')\subseteq A$, $q(B')\subseteq B$ and $q|_{A'\cap B'}:A'\cap B'\rightarrow A\cap B$ is surjective. We then want to show that
$$
g(P_A, P_B)q=g\mleft(P_A\, q|_{A'}, P_B\, q|_{B'}\mright) .
$$
Evaluating the right-hand side on any $\gamma\in O^{C'}$ results in
\begin{align*}
	g\mleft(P_A\, q|_{A'}, P_B\, q|_{B'}\mright)(\gamma)&=\frac{(P_A q|_{A'})(\gamma|_{A'})\cdot (P_B q|_{B'})(\gamma|_{B'})}{(P_A q|_{A'})|_{A'\cap B'}(\gamma|_{A'\cap B'})}\\[5pt]
&=\frac{\sum\limits_{\alpha\in O^A\text{ s.t. }\alpha\,q|_{A'}=\gamma|_{A'}}P_A(\alpha)\sum\limits_{\beta\in O^B\text{ s.t. }\beta\,q|_{B'}=\gamma|_{B'}}P_B(\beta)}{\sum\limits_{\tau\in O^{A\cap B}\text{ s.t. }\tau\, q|_{A'\cap B'}=\gamma|_{A'\cap B'}}P_{A\cap B}(\tau)}\\[5pt]
&=\sum\limits_{\gamma'\in O^C\text{ s.t. }\gamma' q=\gamma} \frac{ P_A(\gamma'|_{A}) P_B(\gamma'|_{B})}{P_{A\cap B}(\gamma'|_{A\cap B})}=\mleft(g(P_A, P_B)q\mright)(\gamma)
\end{align*}
The third equality uses surjectivity of $q$ on the intersection, which implies that there is at most one $\tau$ satisfying the condition $\tau q|_{A'\cap B'}=\gamma|_{A'\cap B'}$, as well as the sheaf condition to unify the two sums.
\end{proof}

By applying a change of base as defined in Example~\ref{basegleaves} along the inclusion functor $\Delta\hookrightarrow \FinSets$, we obtain a proof of Proposition~\ref{probabilities}.

\subsection{Relational databases}
\label{reldatabase}

The probability distributions example of the previous subsection is formulated in terms of real-valued probabilities. However, all that we have used is that probabilities are elements of a semifield $(K,+,\cdot)$ which additionally satisfies $x+y=0\:\Rightarrow\: x=0$. Consequently, the probability distributions example makes sense over any such $K$. A particularly interesting instance of this is the Boolean semifield $(\{0,1\}, \lor, \land)$: if we interpret a value of $0$ as ``impossible'' and a value of $1$ as ``possible'', a distribution with values in the Boolean semifield can be thought of as a ``possibility distribution''. Since such a distribution is determined by the subset of those outcomes to which it assigns $1$, it can equivalently be regarded as a (non-empty) \emph{relation} whose arity is given by the number of variables involved.

This is essentially what is studied in the theory of \emph{relational databases}~\cite{database}. In order not to be repetitive, we will not discuss the gleaf of possibility distributions in any more detail, but rather explain the intimately related gleaf of relations in the language of database theory, following the exposition of~\cite{Abram2}. 

We fix a set $\mathcal{A}$, thought of as a universe of attributes. In contrast to before, where we assumed all variables to take values in the same set of outcomes, we now allow each $a\in \mathcal{A}$ to take values in a different and possibly infinite set $D_a$. Consider the lattice $2^{\mathcal{A}}$ of finite subsets of $\mathcal{A}$ ordered by inclusion. For a finite set of attributes $A\in 2^{\mathcal{A}}$, an $A$-relation $T_A$ is a subset of $\prod_{a\in A}D_a$. We think of such a $T_A$ as a table in a database whose columns are the attributes $a\in A$.
The assignment
\begin{eqnarray*}
\mathcal{R}:(2^{\mathcal{A}})^{\op}&\rightarrow& \Sets\\
A&\mapsto& \{\text{ $A$-relations }T_A\;\}
\end{eqnarray*}
is a presheaf where the restriction maps are given by the operation of projecting subsets along $\prod_{a\in A}D_a\to\prod_{a\in A'}D_a$ for $A'\subseteq A$. The gluing operation then takes the form
\begin{eqnarray*}
g:\mathcal{R}(A)\times_{\mathcal{R}(A\cap B)}\mathcal{R}(B)&\longrightarrow&\mathcal{R}(A\cup B) \\
(T_A, T_B)&\longmapsto & T_{A\cup B} \defin \left\{r\in\prod_{a\in A\cup B} D_a\;\bigg{|}\;r|_{A}\in T_A,\;r|_{B}\in T_B\right\},
\end{eqnarray*}
which is the \emph{natural join} from the theory of relational databases~\cite{database}.

\begin{Proposition}
The pair $(\mathcal{R}, g)$ is a gleaf. 
\end{Proposition}

A variant of this result---with $2^\mathcal{A}$ replaced by the lattice of finite subsets of $\mathcal{A}$---is due to Dawid and Studen\'y~\cite[Proposition~4.1]{DS}.

\begin{proof}
Condition~\ref{littlegleaf}\ref{littlega} is immediate. For condition~\ref{littlegleaf}\ref{littlegb}, we need to show that the diagram 
\be
\vxymatrix{ \mathcal{R}(A')\times_{\mathcal{R}(A'\cap B)}\mathcal{R}(B) \ar[rr]^{g} \ar[dd]_{g} && \mathcal{R}(A\cup B) \\\\
\mathcal{R}(A\cup B) \ar[rr] && \mathcal{R}(A)\times_{\mathcal{R}(A\cap B)}\mathcal{R}(B) \ar[uu]_{g} 
}\ee
commutes for $A'\subseteq A$ and $A'\cup B=A\cup B$. To this end, we first claim that
\[
g(T_{A'}, T_B)|_{B}=T_B.
\]
In fact, a given $r\in\prod_{a\in B} D_a$ lies in $g(T_{A'}, T_B)|_{B}$ iff there exists $r'\in\prod_{a\in A'\cup B} D_a$ such that $r'|_{B}=r$, $r'|_{A'}\in T_{A'}$ and $r'|_{B}\in T_B$. This is equivalent to requiring $r\in T_B$ and $r|_{A'\cap B}\in T_{A'|A'\cap B}$. By the compatibility condition $T_{A'|A'\cap B}=T_{B|A'\cap B}$, the second condition means that  $r|_{A'\cap B}\in T_{B|A'\cap B}$, which is automatic thanks to $r\in T_B$.

Second, we show that
\[
g(T_{A'}, T_B)|_{A}=\left\{r\in\prod_{a\in A}D_a\;\bigg|\;r|_{A'}\in T_{A'},\;r|_{A\cap B}\in T_{B|A\cap B}\right\}. 
\]
Indeed, a given $r\in  \prod_{a\in A} D_a$ lies in $g(T_{A'}, T_B)|_{A}$ iff there exists $r'\in\prod_{a\in A'\cup B} D_a$ such that $r'|_{A}=r$, $r'|_{A'}\in T_{A'}$ and $r'|_{B}\in T_B$. By $A'\subseteq A$, this is equivalent to requiring that $r|_{A'}\in T_{A'}$ and $r|_{A\cap B}\in T_{B|A\cap B}$, as claimed.

With these observations, we obtain
\begin{align*}
 g\mleft(g(T_{A'}, T_B)|_{A}, g(T_{A'}, T_B)|_{B}\mright)&= g\mleft(\left\{r\in\prod_{a\in A}D_a\;\bigg|\;r|_{A'}\in T_{A'},\;r|_{A\cap B}\in T_{B|A\cap B}\right\}, T_B\mright)\\[5pt]
&= \left\{r\in\prod_{a\in A\cup B}D_a\;\bigg|\;r|_{A'}\in T_{A'},\;r|_{A\cap B}\in T_{B|A\cap B},\; r|_{B}\in T_B\right\}\\[5pt]
&=\left\{r\in\prod_{a\in A\cup B}D_a\;\bigg|\;r|_{A'}\in T_{A'},\; r|_{B}\in T_B\right\}\\[5pt]
&=g(T_{A'}, T_B),
\end{align*}
as desired. A similar proof applies for $B'\subseteq B$ and $A\cup B=A\cup B'$.

We now consider condition \ref{littlegleaf}\ref{littlegc} for $A'\subseteq A$ and $A\cap B=A'\cap B$. We need to show that the diagram 
\[
\xymatrix{ \mathcal{R}(A) \times_{\mathcal{R}(A\cap B)}\mathcal{R}(B) \ar[dd]\ar[rr]^-{g} && \mathcal{R}(A\cup B) \ar[dd]\\\\
\mathcal{R}(A')\times_{\mathcal{R}(A\cap B)} \mathcal{R}(B) \ar[rr]^-{g} && \mathcal{R}(A'\cup B)
}\]
commutes. For a given $r\in \prod_{a\in A'\cup B}D_a$, we have $r\in g(T_{A}, T_B)|_{A'\cup B}$ iff there exists an $r'\in \prod_{a\in A\cup B}D_a$ such that $r'|_{A'\cup B}=r$, $r'|_{A}\in T_A$ and $r'|_{B}\in T_B$. Upon choosing $\widehat{r}=r'|_{A}$, this implies the existence of an $\widehat{r}\in T_A$ such that $\widehat{r}|_{A'}=r|_{A'}$ and $r|_{B}\in T_B$. Conversely, if this condition is satisfied, then $A\cap B=A'\cap B$ guarantees that $r\in  \prod_{a\in A'\cup B}D_a$ can be extended to an $r'\in \prod_{a\in A\cup B}D_a$ having the necessary properties.
The existence of $\widehat{r}$ is equivalent to $r|_{B}\in T_B$ and $r|_{A'}\in T_{A|A'}$.
\end{proof}

\subsection{Topological spaces}
\label{topological}

We expect that many geometrical structures, besides metric spaces, form gleaves of type $\Sets^{\op}\rightarrow\Sets$. One of these is the structure of carrying a topology, and this will be our last example of a gleaf.

Similar to our other examples, we define the presheaf as
\begin{align*}
\mathcal{T}:\Sets^{op}&\rightarrow\Sets\\[5pt]
A&\mapsto\mathcal{T}(A)\defin\left\{\tau\subseteq 2^{A}\;\big|\;\tau\text{ is a topology on } A\right\} \\[5pt]
\mleft(\xymatrix{A\ar[r]^f & B}\mright)&\mapsto\mleft(\vcenter{\vbox{\xymatrix@R=.2cm@C=0cm{ \mathcal{T}(B)\ar[rr] && \mathcal{T}(A) \\ \tau_B\ar@{|->}[rr] && \{f^{-1}(U)\;|\;U\in\tau_B\} }}}\mright).
\end{align*}

This presheaf is not a sheaf since there are triples of pairwise compatible local sections which cannot be consistently joined. In fact, consider three sets $U, V, W$, each with exactly two elements as in Figure~\ref{triangle}, and such that $U$ and $V$ have the indiscrete topology while $W$ has the discrete topology. Clearly, the subspace topologies on pairwise intersections are the same, but it is not possible to find a topology on $U\cup V\cup W$ which restricts to the given topologies on the individual sets: any open containing one point of $W$ has to contain $U\cap V$ as well, and hence also the other point of $W$. Nevertheless, we will show in the following that $\mathcal{T}$ can be turned into a gleaf.

Given the system of bicoverings for $\Sets$ defined in Section~\ref{takeII}, the gluing operation is:
\begin{align*}
g:\mathcal{T}(A)\times_{\mathcal{T}(A\times_C B)}\mathcal{T}(B)&\;\rightarrow\;\mathcal{T}(C)\\[5pt]
(\tau_A, \tau_B)&\;\mapsto\; g(\tau_A, \tau_B)\defin\left\{U\subseteq C\;\big|\; U\cap A\in \tau_A,\;U\cap B\in \tau_B\right\}.
\end{align*}
As defined, $g(\tau_A, \tau_B)$ clearly is a topology on $Z$.

\begin{Proposition}
With these definitions, $(\mathcal{T}, g)$ is a gleaf.
\end{Proposition}

\begin{proof}
The identity axiom~\ref{defbiggleaf}\ref{identityax} holds trivially. Before proceeding with the proof of the other two axioms, we show that the glued topology indeed restricts to the two given ones,
\be
\label{henry}
g(\tau_A,\tau_B)|_{A}=\tau_A,\qquad g(\tau_A,\tau_B)|_{B}=\tau_B.
\ee
It is enough to prove the first equation. The inclusion $g(\tau_A,\tau_B)|_{A}\subseteq\tau_A$ is trivial, so we only need to show the other direction, which crucially relies on the assumed compatibility of the topologies, $\tau_{A|A\cap B}=\tau_{B|A\cap B}$. For any $V\in\tau_A$, we have $V\cap B\in\tau_{A|A\cap B}=\tau_{B|A\cap B}$, so that there exists $W\in\tau_B$ with $W\cap A=V\cap B$. Then putting $U=V\cup W$ gives $U\cap A=(V\cap A)\cup (W\cap A)=V\cup (V\cap B)=V$, and similarly $U\cap B=W$, which shows that $U\in g(\tau_A,\tau_B)$.

For the back-and-forth axiom~\ref{defbiggleaf}\ref{BFax}, we need to show that for any $A'\subseteq A$ with $A'\cup B=A\cup B=C$, the diagram 
\be\label{bftop}\vxymatrix{
\mathcal{T}(A')\times_{\mathcal{T}(A'\times_C B)}\mathcal{T}(B)\ar[r]^-g\ar[d]_g&\mathcal{T}(C)\\
\mathcal{T}(C)\ar[r]&\mathcal{T}(A)\times_{\mathcal{T}(A\times_C B)}\mathcal{T}(B)\ar[u]_g
}\ee
commutes. In a manner similar to the proof of~\eqref{henry}, it can be shown that
\[
g (\tau_{A'}, \tau_B)|_{ A} = \left\{U'\subseteq A\;\big|\; U'\cap A'\in\tau_{A'},\; \exists V\in\tau_B\text{ s.t. }U'\cap B=V\cap A\right\} .
\]
This lets us evaluate the ``long'' composition in~\eqref{bftop} to
\begin{align*}
g\mleft(g (\tau_{A'}, \tau_B)|_{ A}, g (\tau_{A'}, \tau_B)|_{B}\mright)&=g\mleft(\{U'\subseteq A\;\big|\; U'\cap A'\in\tau_{A'},\; \exists V\in\tau_B\text{ s.t. }U'\cap B=V\cap A\},    \tau_B\mright)\\[5pt]
&=\{U\subseteq C\;|\;U\cap A\cap A'\in\tau_{A'},\;\exists \;V\in \tau_B\text{ s.t. } U\cap A\cap B=V\cap A ,\;U\cap B\in\tau_B\}\\[5pt]
&=\{U\subseteq C\;|\;U\cap A'\in\tau_{A'},\; U\cap B\in\tau_{B}\}.
\end{align*}
For the last equality, we use that the inclusion ``$\subseteq $'' is trivial, while for the inverse inclusion ``$\supseteq$'' we simply pick $V\defin U\cap B$. Therefore $g\mleft(g (\tau_{A'}, \tau_B)|_{ A}, g (\tau_{A'}, \tau_B)|_{B}\mright)=g (\tau_{A'}, \tau_B)$, which is commutativity of the above diagram.

In the partial naturality axiom~\ref{defbiggleaf}\ref{natural}, we have another bicovering $C'=A'\cup B'$ and a map $q:C'\rightarrow C$ such that $q(A')\subseteq A$, $q(B')\subseteq B$ and $q|_{A'\cap B'}:A'\cap B'\rightarrow A\cap B$ is surjective. We need to show that the diagram 
\[\xymatrix{
\mathcal{T}(A)\times_{\mathcal{T}(A\times_C B)}\mathcal{T}(B)\ar[rr]^-g\ar[dd]&&\mathcal{T}(C)\ar[dd]\\\\
\mathcal{T}(A')\times_{\mathcal{T}(A'\times_{C'}B')}\mathcal{T}(B')\ar[rr]^-g&&\mathcal{T}(C')
}\]
commutes. The two compositions in this square evaluate to 
\begin{align}\begin{split}\label{split1}
g(\tau_A, \tau_B)q &= \{U\subseteq C\:|\:U\cap A\in\tau_A,\;U\cap B\in \tau_B\}q \\[5pt]
&=\{q^{-1}(U)\:|\:U\subseteq C,\;U\cap A\in\tau_A,\;U\cap B\in \tau_B\}
\end{split}
\end{align}
and 
\begin{align}\begin{split}\label{split2}
g(\tau_A\, q|_{A'},& \tau_B\, q|_{B'})=\{U'\subseteq C'\;|\;U'\cap A'\in \tau_A q|_{A'},\; U'\cap B'\in \tau_B q|_{B'}\}\\[5pt]
&=\{U'\subseteq C'\;|\;\exists V\in\tau_A\text{ s.t. }q^{-1}(V)\cap A'=U'\cap A',\\
&\qquad\qquad\qquad\exists W\in \tau_B\text{ s.t. }q^{-1}(W)\cap B'=U'\cap B'\:\},
\end{split}
\end{align}
respectively.
For a given $U'$ in~\eqref{split2} with associated $V$ and $W$ we obtain, using the surjectivity assumption $q^{-1}(A\cap B)=A'\cap B'$,
$$q^{-1}(W\cap A)=q^{-1}(W\cap A\cap B)=(U'\cap B')\cap (A'\cap B')=U'\cap A'\cap B'=\ldots=q^{-1}(V\cap B),$$ where the dots indicate steps analogous to the first half. Again by surjectivity of $q$ on the intersection, we can ``cancel'' $q^{-1}$, which results in $W\cap A= V\cap B$. Therefore $V\cup W\in g(\tau_A, \tau_B)$, since $(V\cup W)\cap A=V\cup(W\cap A)=V\in \tau_A$ and similarly $ (V\cup W)\cap B=(V\cap B)\cup W=W\in \tau_B$. Moreover,
$$
q^{-1}(V)\cap B'\subseteq q^{-1}(V\cap B) = q^{-1}(W\cap A) = U'\cap A'\cap B',
$$
and similarly $q^{-1}(W)\cap A' \subseteq U'\cap A'\cap B'$. Together with~\eqref{split2}, this is relevant for proving the third equality in
\begin{align*}
q^{-1}(V\cup & W)=q^{-1}(V)\cup q^{-1}(W) \\[5pt]
 &= (q^{-1}(V)\cap A')\cup (q^{-1}(V)\cap B')\cup (q^{-1}(W)\cap A') \cup (q^{-1}(W)\cap B') \\[5pt]
 & = (U'\cap A')\cup(U'\cap B')=U' .
\end{align*}
This shows that $U'$ is indeed an element of in $g(\tau_A, \tau_B)q$.

For the reverse inclusion of~\eqref{split1} in~\eqref{split2}, we start with any $q^{-1}(U)\in g(\tau_A, \tau_B)q$, so that we know that $V=U\cap A\in\tau_A$ and $W=U\cap B\in \tau_B$. Then $q^{-1}(V)\cap A'=q^{-1}(U\cap A)\cap A'=q^{-1}(U)\cap q^{-1}( A)\cap A'=q^{-1}(U)\cap A'$, and similarly $q^{-1}(W)\cap B'=q^{-1}(U)\cap B'$. Therefore $q^{-1}(U)\in g(\tau_A\, q|_{A'},\tau_B\, q|_{B'})$.
\end{proof}

\section{Further directions}\label{further}

Here we would like to present some possible ideas for future research.

\subsection{Compositories}

\paragraph{Nerves of categories and compatibility with face maps.}

As we saw in Section~\ref{nerves}, the nerve of a (small) category is a compository which satisfies compatibility of composition with all face maps, including~\eqref{facenot}, which typically does not hold in other compositories. The question is now if the converse is true: is a compository satisfying~\eqref{facenot} isomorphic to the nerve of a category?

\paragraph{Fundamental compositories and homotopy type theory.}

Recall the definition of the fundamental $\infty$-groupoid $\Pi(X)$, in its Kan complex version: 

\begin{Definition}[e.g.~\cite{model}]
Let $X$ be a topological space. The \emph{fundamental $\infty$-groupoid} $\Pi(X)$ is the simplicial set with $n$-simplices
\[
\Pi(X)(n) \defin \Top(\Delta_n,X),
\]
where $\Delta_n$ is the geometric realisation of the $n$-simplex.
\end{Definition}

Here, the face and degeneracy maps are induced from considering the assignment $[n]\mapsto \Delta_n$ as a functor $\Delta^{\op}\to\Tops^{\op}$. 

It is an interesting question whether $\Pi(X)$ carries the structure of a compository. The most natural way of turning $\Pi(X):\Delta^{\op}\to\Sets$ into a compository would be for the (co-)presheaf $\Delta^{\op}\to\Tops^{\op}$ to carry the structure of a gleaf, i.e.~if there existed a family of continuous maps
\[
\Delta_{m+n-k} \longrightarrow \Delta_m\amalg_{\Delta_k}\Delta_n
\]
satisfying a list of axioms dual to those for compositories.

On a related note, we might consider types in homotopy type theory~\cite{HoTTbook} rather than topological spaces. These types are usually regarded to be globular $\infty$-groupoids~\cite{Lumsd,BG} in the sense of Batanin and Leinster~\cite{Batanin,Leinster2}. Now a natural question is, is there a variant of homotopy type theory in which the types naturally carry the structure of a compository?

First of all, this would require the higher identity terms over a type to form a simplicial set rather than a globular set. Obtaining a simplicial set of higher identity terms will require a departure from the usual ``binary'' identity types $\id_A(a,b)$ taking two arguments only, allowing for the introduction of arbitrary $n$-ary identity types $\id_A(a_1,\ldots,a_n)$; such identity types can also be defined as inductive types generated by reflexivity terms. This leads to the question, what will be the resulting algebraic structure on the collection of all such (higher) identity terms?

\paragraph{Dagger compositories and symmetric compositories.}
In many of our examples of compositories $\sC$, there is a canonical way of assigning to a simplex $S\in\sC(m)$ other simplices of the same dimension which are obtained by permuting the vertices of $S$. For example in a compository of higher spans (Section~\ref{higherspans}), this applies to the ``mirror image'' $S^\dag$ of any higher span $S$ obtained by precomposing $S:\Sp_m\rightarrow\sC$ with the ``reflection''
$$
\Sp_m\rightarrow\Sp_m,\qquad (v,w)\mapsto (m-w,m-v).
$$
In this way, we expect that a compository of higher spans is actually a \emph{dagger compository}, just as an ordinary category of spans is a \emph{dagger category}~\cite{dagnlab}.

In our examples of metric spaces (Section~\ref{metspacesI}) and joint probability distributions (Section~\ref{secjpd}), there is even more symmetry: the vertices of any simplex can be permuted arbitrarily, and hence the symmetric group $S_m$ acts on the set of $m$-simplices $\sC(m)$. This should correspond to a notion of \emph{symmetric compository}, which we expect to be a gleaf of type $\FinSets^{\op}_{\neq\emptyset}\rightarrow\Sets$, where $\FinSets_{\neq\emptyset}$ stands for the category of non-empty finite sets.

\paragraph{Hyperstructures.} The hyperstructures of Baas~\cite{Baas} are mathematical abstractions of collections of systems which may form bonds, e.g.~chemical bonds between atoms, and also bonds between bonds etc. We imagine that compositories may provide an alternative to hyperstructures and/or a related approach to systems that have the potential to form bonds: upon taking the systems to be the $0$-simplices in a compository, and an $n$-system bond to be an $(n-1)$-simplex between its $n$ constituent systems as vertices, we might obtain a composition operation on those bonds which could turn them into a compository.

\paragraph{A self-referential theory of compositories?} The fundamental nature of category theory getsgets  reflected in the self-referential character of category theory, such as the fact that categories themselves form a category. Can one develop a theory of compositories in a similarly self-referential manner? For example, is it possible to define a sensible ``compository of compositories''? Doing this would require one to temporarily forget about category theory completely and to try and adapt a different mindset which moves the focus from \emph{functions} having a given domain and codomain to \emph{relations} of arbitrary finite arity in which the distinction between domain and codomain is blurred.

\subsection{Gleaves}

\paragraph{Gleaves and monoidal structures.} The definition of gleaf~\ref{defbiggleaf} closely resembles that of lax monoidal functors. Can this analogy be made more precise? If monoidal structures on categories generalise binary products, then what generalises pullbacks?

\paragraph{Gluing more than two local sections.} The gluing operation, which is part of the structure of a gleaf, glues pairs of local sections only; and, as we have frequently noted in the examples, three or more local sections can in general not be consistently extended at all. However, when the different sets or objects, on which the given local sections live, intersect in a certain particularly nice way, then a canonical extension is indeed possible and can be constructed from repeated application of the gluing operation in a unique way. The associativity result of Proposition~\ref{assocprop} is a particular example of this. See the work of Vorob'ev~\cite{V} for some ideas of how all this works in the case of the joint probability distributions gleaf $(\mathcal{P},g)$ of Section~\ref{take2}.

We plan to investigate all this in more detail in a follow-up to this article; this should contain in particular a generalisation of Vorob'ev's theorem and a theorem stating how any local section of a gleaf gives rise to a semi-graphoid~\cite{semigraphoids}, also generalising from the well-studied structure of the joint probability distributions example.

\paragraph{Gleaves and fibred categories.} In many of our examples of $\Sets$-valued gleaves, e.g.~in the metric space example (Section~\ref{takeII}), the values of the presheaf have actually more structure than merely being sets; we suspect that they really should be considered as posets or even categories. For example, we may say that two metrics $d_1,d_2$ on a set $X$ satisfy $d_1\leq d_2$ if and only if $d_1(x,y)\leq d_2(x,y)$ for all $x,y\in X$. Then, the gluing operation in the metric space example has a universal property: it arises as the smallest metric, relative to this ordering, which restricts to the given ones on the subsets. Do the gluing operations of all our examples of gleaves have universal properties? If so, does this imply that the notion of gleaf should not be considered fundamental in any sense? 

More generally, one can consider the category $\Mets$ of metric spaces (without the non-degeneracy axiom) and distance-nonincreasing functions as a fibred category over $\Sets$~\cite{Vistoli}. How does this relate to the gleaf of metric spaces? Under which conditions does a fibred category give rise to a gleaf?

\paragraph{Enriched categories as gleaves\footnote{This question is due to Jonathan Elliott.}.} On a related note, upon regarding metric spaces (without the symmetry and non-degeneracy axioms) as enriched categories~\cite{Lawv}, the putative universal property of the gluing operation is likely to correspond to a certain (co-)limit in the symmetric monoidal base category $\mleft([0,\infty),\geq,+\mright)$. Can this be generalised to enrichment over an arbitrary suitably \mbox{(co-)complete} symmetric monoidal small base category $\mathcal{V}$? More precisely, let $\Gamma_{\mathcal{V}}:\Sets^{\op}\to\Sets$ be the presheaf which assigns to every set $X$ the set of all $\mathcal{V}$-categories having $X$ as their underlying set of objects. Can this presheaf be turned into a gleaf in a natural way?

\paragraph{Gleaves as models of a sketch.} The proof of Theorem~\ref{johnstonetheorem} exhibits a gleaf over the base category considered there as a model of a sketch~\eqref{sketch}. Can this construction be generalised for gleaves over any small category with bicoverings?
Alternatively, we could also ask whether the category of gleaves over a small base is accessible~\cite{accessible}.


\begin{bibdiv}
\begin{biblist}

\bib{database}{book}{
author = {Abiteboul, Serge},
author = {Hull, Richard},
author = {Vianu, Victor},
title = {Foundations of Databases},
publisher = {Addison Wesley},
year = {1995},
}

\bib{Abram}{incollection}{
author = {Abramsky, Samson},
title = {Retracing some paths in process algebra},
booktitle = {C{ONCUR} '96: concurrency theory {{P}isa}},
series = {Lecture Notes in Comput. Sci.},
volume = {1119},
pages = {1--17},
publisher = {Springer},
year = {1996},
}

\bib{Abram2}{misc}{
author = {Abramsky, Samson},
title = {Relational Databases and {B}ell's Theorem},
booktitle = {In search of elegance in the theory and practice of computation},
series = {Lecture Notes in Comput. Sci.},
volume = {8000},
pages = {13--35},
publisher = {Springer, Heidelberg},
year = {2013},
}

\bib{AB}{article}{
  author={Abramsky, Samson},
 author = {Brandenburger, Adam},
  title={The sheaf-theoretic structure of non-locality and contextuality},
  journal={New Journal of Physics},
  volume={13},
  number={11},
  pages={113036},
  year={2011},
}

\bib{accessible}{book}{
author={ Ad\'{a}mek, Ji\v{r}\'{i}},
author={Rosick\'{y}, Ji\v{r}\'{i}},
title={Locally presentable and accessible categories},
publisher={Cambridge University Press},
year={1994}
}


\bib{Baas}{article}{
author = {Baas, Nils},
title = {On structure and organization: an organizing principle},
journal = {International Journal of General Systems},
volume = {42},
number = {2},
year = {2013},
}

\bib{clark}{misc}{
author={Barwick, Clark},
author={Schommer-Pries, Christopher},
title={On the Unicity of the Homotopy Theory of Higher Categories},
note={\href{http://arxiv.org/abs/1112.0040}{arXiv:1112.0040}},
year={2011},
}






\bib{Batanin}{article}{
author = {Batanin, Michael},
title = {Monoidal globular categories as a natural environment for the theory of weak $n$-categories},
journal = {Advances in Mathematics},
volume = {136},
number = {1},
pages = {39--103},
year = {1998},
}

\bib{Benabou}{incollection}{
author = {B{\'e}nabou, Jean},
title = {Introduction to bicategories},
booktitle = {Reports of the {M}idwest {C}ategory {S}eminar},
pages = {1--77},
publisher = {Springer},
address = {Berlin},
year = {1967},
}

\bib{BG}{article}{
author = {van den Berg, Benno},
author = {Garner, Richard},
title = {Types are weak $\omega$-groupoids},
journal = {Proc. London Math. Soc.},
volume = {3},
number = {102},
pages = {370--394},
year = {2011},
}

\bib{DS}{incollection}{
author = {Dawid, Alexander Philip},
author = {Studen{\'y}, Milan},
title = {Conditional products: An alternative approach to conditional independence},
booktitle = {Artificial Intelligence and Statistics '99},
year = {1999},
pages = {32--40},
}

\bib{Dob}{incollection}{
author = {Doberkat, Ernst-Erich},
title = {The converse of a stochastic relation},
booktitle = {Foundations of software science and computation structures},
series = {Lecture Notes in Comput. Sci.},
volume = {2620},
pages = {233--249},
publisher = {Springer},
year = {2003},
}

\bib{Engenes}{article}{
author = {Engenes, Hans},
title = {Subobject classifiers and classes of subfunctors},
journal = {Math. Scand.},
volume = {84},
year = {1974},
pages = {145--152},
}

\bib{Fried}{article}{
author = {Friedman, Greg},
title = {Survey article: an elementary illustrated introduction to simplicial sets},
journal = {Rocky Mountain J. Math.},
volume = {42},
number = {2},
pages = {353--423},
year = {2012},
}

\bib{FC}{article}{
author = {Fritz, Tobias},
author = {Chaves, Rafael},
title = {Entropic inequalities and marginal problems},
journal = {IEEE Trans. Inform. Theory},
volume = {59},
number = {2},
pages = {803--817},
year = {2013},
}

\bib{semigraphoids}{article}{
author = {Geiger, Dan},
author = {Pearl, Judea},
title = {Logical and algorithmic properties of conditional independence and graphical models},
journal = {The Annals of Statistics},
volume = {21},
number = {4},
year = {1993},
pages = {2001--2021},
}

\bib{carres}{article}{
author={Guitart, Ren{\'e}},
title={Relations et Carr{\'e}s Exacts},
journal={Ann. sc. math. Qu{\'e}bec},
volume={IV},
number={2},
pages={103--125},
year={1980},
}

\bib{grandis}{article}{
author = {Grandis, Marco},
title = {Higher cospans and weak cubical categories (Cospans in algebraic topology, I)},
journal = {Theory Appl. Categ.},
volume = {18},
number = {12},
pages = {321--347},
year = {2007},
}

\bib{haugseng}{misc}{
author = {Haugseng, Rune},
title = {Iterated spans and ``classical'' topological field theories},
note = {\href{http://arxiv.org/abs/1409.0837}{arXiv:1409.0837}},
year = {2014},
}

\bib{model}{book}{
author = {Hovey, Mark},
title = {Model Categories},
publisher = {American Mathematical Society},
year = {1999},
series={Mathematical Surveys and Monographs},
volume={63}
}

\bib{stone}{book}{
author = {Johnstone, Peter T.},
title = {Stone Spaces},
publisher = {Cambridge University Press},
year = {1982},
}

\bib{Johnstone1}{article}{
author={Johnstone, Peter T.},
title={Collapsed toposes and cartesian closed varieties},
journal={J. Algebra},
volume={129},
pages={446--480},
year={1990}
}

\bib{elephant}{book}{
author = {Johnstone, Peter T.},
title = {Sketches of an Elephant: A Topos Theory Compendium I, II},
publisher = {Oxford University Press},
year = {2002},
}


\bib{Joyal}{article}{
author = {Joyal, Andr{\'e}},
title = {Quasi-categories and {K}an complexes},
note = {Special volume celebrating the 70th birthday of Professor Max Kelly},
journal = {J. Pure Appl. Algebra},
volume = {175},
number = {1-3},
pages = {207--222},
year = {2002},
}

\bib{kv}{misc}{
author = {Kissinger, Aleks},
author = {Vicary, Jamie},
title = {Globular},
note = {A proof assistant for higher category theory, based on a definition of semistrict $4$-category sketched at \href{https://ncatlab.org/nlab/show/Globular\#singularities}{https://ncatlab.org/nlab/show/Globular\#singularities} and \href{https://nforum.ncatlab.org/discussion/6829/globular/?Focus=56460\#Comment_56460}{https://nforum.ncatlab.org/discussion/6829/globular/?Focus=56460\#Comment\_56460}, retrieved on 02/10/2016.},
year = {2016},
}

\bib{adhesive}{incollection}{
author = {Lack, Stephen},
author = {Soboci{\'n}ski, Pawe{\l}},
title = {Adhesive categories},
booktitle = {Foundations of software science and computation structures},
series = {Lecture Notes in Comput. Sci.},
volume = {2987},
pages = {273--288},
publisher = {Springer, Berlin},
year = {2004},
}

\bib{Lawv}{article}{
author = {Lawvere, F.~William},
title = {Metric spaces, generalized logic, and closed categories},
journal = {Reprints in Theory and Applications of Categories},
number = {1},
pages = {1--37},
year = {2002},
}

\bib{Leinster2}{article}{
author = {Leinster, Tom},
title = {A survey of definitions of $n$-category},
journal = {Theory Appl. Categ.},
volume = {10},
number = {1},
pages = {1--70},
year = {2002},
}

\bib{Leinster}{book}{
author = {Leinster, Tom},
title = {Higher operads, higher categories},
series = {London Mathematical Society Lecture Note Series},
volume = {298},
publisher = {Cambridge University Press},
year = {2004},
pages = {xiv+433},
}

\bib{Lumsd}{article}{
author = {Lumsdaine, Peter LeFanu},
title = {Weak $\omega$-categories from intensional type theory},
journal = {Logical Methods in Computer Science},
volume = {6},
number = {3:24},
pages = {1--19},
year = {2010},
}

\bib{MacLane}{book}{
author = {Mac Lane, Saunders},
title = {Categories for the working mathematician},
series = {Graduate Texts in Mathematics},
volume = {5},
edition = {Second edition},
publisher = {Springer-Verlag},
address = {New York},
year = {1998},
}

\bib{MM}{book}{
author = {Mac Lane, Saunders},
author = {Moerdijk, Ieke},
title = {Sheaves in geometry and logic},
series = {Universitext},
note = {A first introduction to topos theory, Corrected reprint of the 1992 edition},
publisher = {Springer-Verlag},
address = {New York},
year = {1994},
}

\bib{dagnlab}{misc}{
author = {nLab},
title = {dagger-category},
note = {\href{http://ncatlab.org/nlab/show/dagger-category}{ncatlab.org/nlab/show/dagger-category}, retrieved on 88/08/2013.},
}

\bib{exsqnlab}{misc}{
author={nLab},
title={exact square},
note={\href{http://ncatlab.org/nlab/show/exact+square\#characterization_15}{ncatlab.org/nlab/show/exact+square\#characterization{\textunderscore}15}, retrieved on 15/08/2013. },
}

\bib{Pan}{incollection}{
author = {Panangaden, Prakash},
title = {The category of {M}arkov kernels},
booktitle = {P{ROBMIV}'98: {F}irst {I}nternational {W}orkshop on {P}robabilistic {M}ethods in {V}erification ({I}ndianapolis, {IN})},
series = {Electron. Notes Theor. Comput. Sci.},
volume = {22},
pages = {17 pp. (electronic)},
publisher = {Elsevier},
year = {1999},
}

\bib{Riehl}{misc}{
author = {Riehl, Emily},
title = {A leisurely introduction to simplicial sets},
note = {\href{http://www.math.jhu.edu/~eriehl/ssets.pdf}{math.jhu.edu/$\sim$eriehl/ssets.pdf}, retrieved on 20/10/2016.},
}

\bib{Scott}{incollection}{
author = {Scott, Dana},
title = {Outline of a mathematical theory of computation},
booktitle = {Technical Monograph {PRG}-2},
publisher = {Oxford University Computing Laboratory},
year = {1970},
}

\bib{Segal}{article}{
author = {Segal, Graeme},
title = {Classifying spaces and spectral sequences},
journal = {Inst. Hautes \'Etudes Sci. Publ. Math.},
number = {34},
year = {1968},
pages = {105--112},
}

\bib{Steiner}{article}{
author = {Steiner, Richard},
title = {The algebra of the nerves of omega-categories},
journal = {Theory Appl. Categ.},
volume = {28},
number = {23},
pages = {733--779},
year = {2013},
}

\bib{Street}{article}{
author = {Street, Ross},
title = {The algebra of oriented simplexes},
journal = {J. Pure Appl. Alg.},
volume = {49},
year = {1987},
pages = {283--335},
}

\bib{HoTTbook}{book}{
author = {Univalent Foundations Program, The},
title = {Homotopy Type Theory},
publisher = {Institute for Advanced Study},
year = {2013},
}

\bib{Verity}{article}{
author = {Verity, Dominic},
title = {Complicial sets characterising the simplicial nerves of strict $\omega$-categories},
journal = {Mem. Amer. Math. Soc.},
volume = {193},
number = {905},
pages = {xvi+184 pp.},
year = {2008},
}

\bib{Vistoli}{misc}{
author = {Vistoli, Angelo},
title = {Notes on Grothendieck topologies, fibered categories and descent theory},
note = {\href{http://arxiv.org/abs/math.AG/0412512}{arXiv:math.AG/0412512}},
year = {2004},
}

\bib{V}{article}{
author = {Vorob'ev, N. N.},
title = {Consistent families of measures and their extensions},
publisher = {SIAM},
year = {1962},
journal = {Theory of Probability and its Applications},
volume = {7},
number = {2},
pages = {147-163},
}

\end{biblist}
\end{bibdiv}
\end{document}